\documentclass[reqno]{amsart}
\usepackage{amsmath}
\usepackage{amsxtra}
\usepackage{amscd}
\usepackage{amsthm}
\usepackage{amsfonts}
\usepackage{amssymb}
\usepackage{enumitem}
\usepackage{microtype}
\usepackage{graphicx}
\usepackage{epstopdf}
\usepackage{mathrsfs}
\usepackage{eucal} 
\usepackage{comment}
\usepackage[hidelinks, pagebackref, bookmarks=false]{hyperref}
\hypersetup{
    colorlinks,
    citecolor=black,
    filecolor=black,
    linkcolor=blue,
    urlcolor=black,
}
\usepackage{cite}
\usepackage{microtype}
\usepackage{url}

\usepackage{thmtools} 
\usepackage{thm-restate}

\input xy
\xyoption{all}

\numberwithin{equation}{subsection} 
\numberwithin{figure}{subsection} 

\makeatletter
\let\c@equation\c@figure
\makeatother

\newtheorem{theorem}[equation]{Theorem}
\newtheorem{corollary}[equation]{Corollary}
\newtheorem{lemma}[equation]{Lemma}
\newtheorem{prop}[equation]{Proposition}
\newtheorem{propconstr}[equation]{Proposition-Construction}

\theoremstyle{remark}
\newtheorem{remark}[equation]{Remark}
\newtheorem{example}[equation]{Example}
\newtheorem{question}[equation]{Question}

\theoremstyle{definition}
\newtheorem{defn}[equation]{Definition}

\newcommand\nc{\newcommand}
\nc\on{\operatorname}

\newcommand\fp{{\mathfrak p}}
\newcommand\fq{{\mathfrak q}}
\newcommand\fm{{\mathfrak m}}

\nc\Hom{\on{Hom}}
\nc\Sections{\on{Sections}}
\nc\Sym{\on{Sym}}
\nc\Spec{\on{Spec}}
\nc\Specm{\on{Specm}}
\nc\ul{\underline}
\nc{\dfp}{\overset{\cdot}{\fp}}
\nc{\dfq}{\overset{\cdot}{\fq}}
\nc{\dfm}{\overset{\cdot}{\fm}}

\begin{document}

\title{Arithmetic representations of fundamental groups I}
\author{Daniel Litt}

\begin{abstract}
Let $X$ be a normal algebraic variety over a finitely generated field $k$ of characteristic zero, and let $\ell$ be a prime.  Say that a continuous $\ell$-adic representation $\rho$ of $\pi_1^{\text{\'et}}(X_{\bar k})$ is \emph{arithmetic} if there exists a representation $\tilde \rho$ of a finite index subgroup of  $\pi_1^{\text{\'et}}(X)$, with $\rho$ a subquotient of $\tilde\rho|_{\pi_1(X_{\bar k})}$. We show that there exists an integer $N=N(X, \ell)$ such that every nontrivial, semisimple arithmetic representation of $\pi_1^{\text{\'et}}(X_{\bar k})$ is nontrivial mod $\ell^N$. As a corollary, we prove that any nontrivial semisimple representation of $\pi_1^{\text{\'et}}(X_{\bar k})$ which \emph{arises from geometry} is nontrivial mod $\ell^N$.
\end{abstract}
\maketitle

\setcounter{tocdepth}{1}
\tableofcontents


\section{Introduction}

\subsection{Statement of main results}
The purpose of this paper is to study representations of the \'etale fundamental group of a variety $X$ over a finitely generated field $k$ --- in particular, we wish to understand the restrictions placed on such representations by the action  of the absolute Galois group of $k$ on $\pi_1^{\text{\'et}}(X_{\bar k})$. 
\begin{defn}\label{arithmetic-defn}
Let $k$ be a field and $X$ a geometrically connected $k$-variety with a geometric point $\bar x$. Then we say that a continuous representation of the geometric \'etale fundamental group of $X$ $$\rho: \pi_1^{\text{\'et}}(X_{\bar k}, \bar x)\to GL_n(\mathbb{Z}_\ell)$$  is \emph{arithmetic} if there exists a finite extension $k'/k$ and a representation $$\tilde{\rho}: \pi_1^{\text{\'et}}(X_{k'}, \bar x)\to GL_n(\mathbb{Z}_\ell),$$ such that $\rho$ is a subquotient of $\tilde{\rho}|_{\pi_1^{\text{\'et}}(X_{\bar k}, \bar x)}$.
\end{defn}
The main theorem of this paper is:
\begin{theorem}\label{main-arithmetic-result}
Let $k$ be a finitely generated field of characteristic zero, and $X/k$ a normal, geometrically connected variety.  Let $\ell$ be a prime.  Then there exists $N=N(X,\ell)$ such that any arithmetic representation $$\rho: \pi_1^{\text{\'et}}(X_{\bar k})\to GL_n(\mathbb{Z}_\ell),$$ which is trivial mod $\ell^N$, is unipotent.
\end{theorem}
This theorem implies the result of the abstract (that any non-trivial arithmetic representation $\rho$ with $\rho\otimes \mathbb{Q}_\ell$ semisimple is non-trivial mod $\ell^N$), because any semisimple unipotent representation is trivial.  Note that $N$ is independent of $n$ (the dimension of the representation $\rho$); to our knowledge this was not expected.
\begin{defn}\label{geometric-defn}
Let $k$ be a field and $X$ a geometrically connected $k$-variety with geometric point $\bar x$. Then we say that a continuous representation $$\rho: \pi_1^{\text{\'et}}(X_{\bar k}, \bar x)\to GL_n(\mathbb{Z}_\ell)$$ is \emph{geometric} if there exists a smooth proper morphism $\pi: Y\to X$ and an integer $i\geq 0$ so that $\rho$ appears as a subquotient of the natural monodromy representation $$\pi_1^{\text{\'et}}(X_{\bar k}, \bar x)\to GL((R^i\pi_*\underline{\mathbb{Z}_\ell})_{\bar x}).$$
\end{defn}
As a corollary of Theorem \ref{main-arithmetic-result}, we have:
\begin{corollary}\label{main-geometric-cor}
Let $k$ be any field of characteristic zero, and $X/k$ a normal, geometrically connected variety.  Let $\ell$ be a prime.  Then there exists $N=N(X, \ell)$ such that any geometric representation $$\rho: \pi_1^{\text{\'et}}(X_{\bar k})\to GL_n(\mathbb{Z}_\ell),$$ which is trivial mod $\ell^N$, is unipotent.
\end{corollary}
If $k=\mathbb{C}$, we may by standard comparison results replace $\pi_1^{\text{\'et}}(X_{\bar k})$ with the usual topological fundamental group $\pi_1(X(\mathbb{C})^{\text{an}})$ in Corollary \ref{main-geometric-cor} above.  For fixed $n$ (the dimension of $\rho$) this corollary (but not Theorem \ref{main-arithmetic-result}) should follow from the main result of \cite{delignemonodromy}; the independence of $N$ from $n$ appears to be new and is, we believe, surprising.

In many cases, the invariant $N(X, \ell)$ from Theorem \ref{main-arithmetic-result} may be made explicit.  For example, if $X=\mathbb{P}^1_k\setminus\{x_1, \cdots, x_m\}$, $N(X, \ell)=1$ for almost all $\ell$; see Section \ref{bounds-section} for details. See Section \ref{sharp-section} for a discussion of the extent to which these results are sharp, and examples of representations which we now know (using Theorem \ref{main-arithmetic-result}) not to be arithmetic or geometric.
\subsection{Theoretical aspects and proofs}
The proof of Theorem \ref{main-arithmetic-result} is ``anabelian" in nature; this is natural as Theorem \ref{main-arithmetic-result} is a purely group-theoretic statement about \'etale fundamental groups of varieties over finitely-generated fields.  We think the fact that these anabelian methods have concrete geometric applications (e.g.~Corollary \ref{main-geometric-cor})  is surprising and interesting.

We now sketch the idea of the proof of Theorem \ref{main-arithmetic-result}.  The proof requires several technical results on arithmetic fundamental groups which are of independent interest.  For simplicity we assume $X$ is a smooth, geometrically connected curve over a finitely generated field $k$; indeed, one can immediately reduce to this case using an appropriate Lefschetz theorem.

\emph{Step 1  (Section \ref{pi1-preliminaries}).}  Let $x\in X(k)$ be a rational point, and choose an embedding $k\hookrightarrow \bar k$; let $\bar x$ be the associated geometric point of $x$.  Let $\pi_1^{\ell}(X_{\bar k}, \bar x)$ be the pro-$\ell$ completion of the geometric \'etale fundamental group of $X$.  Let $\mathbb{Q}_\ell[[\pi_1^{\ell}(X_{\bar k}, \bar x)]]$ be the $\mathbb{Q}_\ell$-Mal'cev Hopf algebra associated to $\pi_1^\ell(X_{\bar k}, \bar x)$ (this is an algebra whose continuous representations are the same as the unipotent $\mathbb{Q}_\ell$-representations of $\pi_1^{\ell}(X_{\bar k}, \bar x)$), and $W^\bullet$ the weight filtration on $\mathbb{Q}_\ell[[\pi_1^{\ell}(X_{\bar k}, \bar x)]]$.  Then for any $\alpha\in \mathbb{Z}_\ell^{\times}$ sufficiently close to $1$ we construct (Theorem \ref{quasi-scalar-pi1}) certain special elements $\sigma_\alpha\in G_k$ which act on $\on{gr}^{-i}_{W}\mathbb{Q}_\ell[[\pi_1^{\ell}(X_{\bar k}, \bar x)]]$ via $\alpha^i\cdot \on{Id}$.   A key step in the construction of such a $\sigma_\alpha$ (Lemma \ref{quasi-scalar-h1}) was suggested to the author by Will Sawin \cite{sawin} after reading an earlier version of this paper, and is related to ideas of Bogomolov \cite{bogomolov}.

The key input here is a semi-simplicity result (Theorem \ref{quasi-scalar-semisimple}), which is likely of independent interest.  Arguments analogous to those of the proof of Theorem \ref{quasi-scalar-semisimple} prove Theorem \ref{frobenius-semisimple}: that if $Y$ is a smooth variety over $\mathbb{F}_q$, admitting a simple normal crossings compactification, then for $y\in Y(\mathbb{F}_q)$, Frobenius acts semisimply on $\mathbb{Q}_\ell[[\pi_1^{\ell}(Y_{\overline{\mathbb{F}_q}}, \bar y)]]$.

\emph{Step 2 (Section \ref{integral-l-adic-periods-section}).}  For each real number $r>0$, we construct certain Galois-stable normed subalgebras $$\mathbb{Q}_\ell[[\pi_1^{\ell}(X_{\bar k}, \bar x)]]^{\leq \ell^{-r}}\subset \mathbb{Q}_\ell[[\pi_1^{\ell}(X_{\bar k}, \bar x)]]$$ called ``convergent group rings," which have the following property:  If $$\rho:\pi_1^{\ell}(X_{\bar k}, \bar x)\to GL_n(\mathbb{Z}_\ell)$$ is trivial mod $\ell^m$ with $r<m$, then there is a natural commutative diagram of continuous ring maps
$$\xymatrix{
\mathbb{Z}_\ell[[\pi_1^{\ell}(X_{\bar k}, \bar x)]]\ar[r]^-{\rho} \ar[d] & \mathfrak{gl}_n(\mathbb{Z}_\ell)\ar[d] \\
\mathbb{Q}_\ell[[\pi_1^{ \ell}(X_{\bar k}, \bar x)]]^{\leq\ell^{-r}} \ar[r] & \mathfrak{gl}_n(\mathbb{Q}_\ell).
}$$
where $\mathbb{Z}_\ell[[\pi_1^{\ell}(X_{\bar k}, \bar x)]]$ is the group ring of the pro-$\ell$ group $\pi_1^{ \ell}(X_{\bar k}, \bar x)$, and $\mathfrak{gl}_n(R)$ is the ring of $n\times n$ matrices with entries in $R$ (Proposition \ref{convergence-prop}).  Our second main result on the structure of the Galois action on $\pi_1^{ \ell}(X_{\bar k}, \bar x)$ is Theorem \ref{quasi-scalar-eigenvectors}, which states that for $\sigma_\alpha$ as in Step 1, there exists $r_\alpha>0$ such that $\mathbb{Q}_\ell[[\pi_1^\ell(X_{\bar k}, \bar x)]]^{\leq \ell^{-r}}$ admits a set of $\sigma_\alpha$-eigenvectors with dense span, as long as $r>r_\alpha$.  Loosely speaking, this means that the denominators of the $\sigma_\alpha$-eigenvectors in $$\mathbb{Q}_\ell[[\pi_1^{\ell}(X_{\bar k}, \bar x)]]/\mathscr{I}^n$$ do not grow too quickly in $n$, where $\mathscr{I}$ is the augmentation ideal.

\emph{Step 3 (Section \ref{applications-section}).} Choose $\sigma_\alpha$ as in Step 1, with $\alpha$ not a root of unity. Suppose $$\rho: \pi_1^{\ell}(X_{\bar k}, \bar x) \to GL(V)$$ is an arithmetic representation on a finite free $\mathbb{Z}_\ell$-module $V$. Then by a socle argument (Lemma \ref{subquotient-lemma}), we may assume that $\rho$ extends to a representation of $\pi_1^{\text{\'et}, \ell}(X_{k'}, \bar x)$ for some $k'/k$ finite.  In particular, for $m$ such that $\sigma_\alpha^m\in G_{k'}\subset G_k$, $\sigma_\alpha^m$ acts on $\mathfrak{gl}(V)$ so that the morphism $$\rho: \mathbb{Z}_\ell[[\pi_1^{ \ell}(X_{\bar k}, \bar x)]]\to \mathfrak{gl}(V)$$ is $\sigma_\alpha^m$-equivariant.

Let $r_\alpha$ be as in Step 2, and suppose $\rho$ is trivial mod $\ell^n$ for some $n>r_\alpha$; choose $r$ with $r_\alpha<r<n$.  Thus by Step 2, we obtain a $\sigma_\alpha^m$-equivariant map $$\tilde{\rho}: \mathbb{Q}_\ell[[\pi_1^{\ell}(X_{\bar k}, \bar x)]]^{\leq \ell^{-r}}\to \mathfrak{gl}(V\otimes \mathbb{Q}_\ell).$$  But $\mathfrak{gl}(V\otimes \mathbb{Q}_\ell)$ is a finite-dimensional vector space; thus the action of $\sigma_\alpha^m$ on $\mathfrak{gl}(V\otimes \mathbb{Q}_\ell)$ has only finitely many eigenvalues.  But by Step 2, and our choice of $\sigma_\alpha$, this implies that for $N\gg 0$, $$\tilde{\rho}(W^{-N}\mathbb{Q}_\ell[[\pi_1^{\ell}(X_{\bar k}, \bar x)]]^{\leq\ell^{-r}})=0,$$ where again $W^{\bullet}$ denotes the weight filtration.  It is not hard to see that the $W$-adic topology and the $\mathscr{I}$-adic topology on $\mathbb{Q}_\ell[[\pi_1^{\ell}(X_{\bar k}, \bar x)]]^{\leq \ell^{-r}}$ agree, where $\mathscr{I}$ is the augmentation ideal. Hence for some $N'\gg 0$, $$\tilde{\rho}(\mathscr{I}^{N'}\cap \mathbb{Q}_\ell[[\pi_1^{\ell}(X_{\bar k}, \bar x)]]^{\leq \ell^{-r}})=0,$$ which implies $\rho$ is unipotent.
\subsection{Comparison to existing results} 

This work was motivated by the geometric torsion conjecture (see e.g.~\cite{Cadoret2011}), but is of a rather different nature than most previous results. To our knowledge, most prior results along these lines employ complex-analytic techniques, which began with Nadel \cite{nadel} and were built on by Noguchi \cite{noguchi} and Hwang-To \cite{hwang-to}.  There has been a recent flurry of interest in this subject, notably two recent beautiful papers by Bakker-Tsimerman \cite{bak2, bak1}, which alerted the author to this subject.  See also the paper of Brunebarbe \cite{brunebarbe}.

These beautiful results all use the hyperbolicity of $\mathscr{A}_{g,n}$ (the moduli space of principally polarized Abelian varieties with full level $n$ structure) to obstruct maps from curves of low gonality (of course, this description completely elides the difficult and intricate arguments in those papers).  The paper \cite{hwang-to} also proves similar results for maps into other locally symmetric spaces.

The papers \cite{nadel, noguchi, hwang-to} together show that  there exists an integer $N=N(g, d)$ such that if $A$ is a $g$-dimensional Abelian variety over a curve $X/\mathbb{C}$ with gonality $d$, then $A$ cannot have full $N$-torsion unless it is a constant Abelian scheme.  The main deficiency of our result in comparison to these is that we do not give any uniformity in the gonality of $X$.  For example, if $X=\mathbb{P}^1\setminus \{x_1, \cdots, x_n\}$, our results (for example Theorem \ref{p1-main-theorem}) depend on the cross-ratios of the $x_i$.  Of course Example \ref{ridic-example} shows that such a dependence is necessary.

Moreover, our result is uniform in $g$, whereas the best existing results (to our knowledge) are at least quadratic in $g$.  Thus for any given $X$, our results improve on those in the literature for $g$ large.  To our knowledge, this sort of uniformity in $g$ was not previously expected, and is quite interesting.  Moreover in many cases our result is sharp.  See Section \ref{sharp-section} for further remarks along these lines.

Our results also hold for arbitrary representations which arise from geometry, rather than just those of weight $1$ (e.g.~those that arise from Abelian varieties).  

Finally, to our knowledge, Theorem \ref{p1-main-theorem} is the first result (apart from \cite{poonen}, which bounds torsion of elliptic curves over function fields) along these lines which works in positive characteristic.  

More arithmetic work in this subject has been done by Abramovich-Madapusi Pera-Varilly-Alvarado \cite{abr1}, and Abramovich-Varilly-Alvarado \cite{abr2}, which relate the geometric versions of the torsion conjecture to the arithmetic torsion conjecture, assuming various standard conjectures relating hyperbolicity and rational points. Cadoret and Cadoret-Tamagawa have also proven beautiful related arithmetic results (see e.g.~\cite{cadoret-tamagawa, cadorettamagawa2, cadorettamagawa3, cadorettamagawa4}).  See also e.g.~Ellenberg, Hall, and Kowalski's beautiful paper \cite{ellenberg}.

The technical workhorse of this paper is a study of the action of the Galois group of a finitely generated field $k$ on the geometric fundamental group of a variety $X/k$.  We now compare our results on this subject to those in the literature.  Deligne \cite[Section 19]{deligne} studies the action of $\on{Gal}(\overline{\mathbb{Q}}/\mathbb{Q})$ on a certain $\mathbb{Z}_\ell$-Lie algebra associated to a metabelian quotient of $\pi_1^{\text{\'et}}(\mathbb{P}^1\setminus \{0,1,\infty\})$.  In particular, he shows that certain ``polylogarithmic" extension classes are torsion, with order given by the valuations of special values of the Riemann zeta function at negative integers.  Our results (in Section \ref{integral-l-adic-periods-section}) are much blunter than these.  On the other hand, we do give results for the entire fundamental group, rather than a (finite rank) metabelian quotient.  Thus we are able to study \emph{integral, non-unipotent} aspects of the representation theory of fundamental groups.

Certain aspects of this work are also intimately related to the so-called $\ell$-adic iterated integrals of Wojtkowiak \cite{wojtkowiakI, wojtkowiakII, wojtkowiakIII, wojtkowiakIV, wojtkowiakV}.  In Section \ref{integral-l-adic-periods-section}, we bound certain ``integral $\ell$-adic periods"; these are related to Wojtkowiak's $\ell$-adic multiple polylogarithms. They are also $\ell$-adic analogues of Furusho's $p$-adic multiple zeta values \cite{furosho1, furosho2}.  

One may also view this work as an application of anabelian geometry to Diophantine questions (in particular, about function-field valued points of $\mathscr{A}_{g,n}$ and other period domains).  There is a tradition of such applications---see e.g. Kim's beautiful paper \cite{kim} and work of Wickelgren \cite{wickelgren, wickelgren2}, for example.  In particular, we believe our ``integral $\ell$-adic periods" to be related to Wickelgren's work on Massey products.

We believe the main contributions of this paper to be:
\begin{enumerate}
\item the introduction of the rings $\mathbb{Q}_\ell[[\pi_1^{\ell}(X_{\bar k}, \bar x)]]^{\leq \ell^{-r}},$ and
\item the study of the action of $G_k$, for $k$ a finitely generated field, on these rings.
\end{enumerate}
We also believe the applications of anabelian methods to questions about monodromy is interesting in and of itself --- the only antecedent of which we are aware is Grothendieck's proof of the quasi-unipotent monodromy theorem \cite[Appendix]{serre-tate}.
\subsection{Acknowledgments} 

To be added after the referee process is complete.


\section{Preliminaries on fundamental groups and some associated rings}\label{pi1-preliminaries}
\subsection{Basic definitions}
Let $k$ be a finitely-generated field of characteristic zero, and $\bar k$ an algebraic closure of $k$.  Let $X$ be a smooth, geometrically connected variety over $k$ and $\overline{X}$ a simple normal crossings compactification.  Such an $\overline{X}$ exists by Hironaka \cite{hironaka}.  Let $D=\bigcup_i D_i=\overline{X}\setminus X$ be the boundary, with $D_i$ the irreducible components of $D$.   Let $\ell$ be a prime.  We fix this notation for the rest of the paper.

If $\bar x$ is a geometric point of $X$, we denote by $\pi_1^{\text{\'et}}(X_{\bar k}, \bar x)$ the geometric \'etale fundamental group of $X$, and $\pi_1^\ell(X_{\bar k}, \bar x)$ its pro-$\ell$ completion.  If $\bar x_1, \bar x_2$ are geometric points of $X$, we let $\pi_1^{\text{\'et}}(X_{\bar k}; \bar x_1, \bar x_2)$ denote the torsor of \'etale paths between $\bar x_1, \bar x_2$, i.e. the pro-finite set of isomorphisms between the fiber functors associated to $\bar x_1, \bar x_2$ (see e.g. \cite[Collaire 5.7 and the surrounding remarks]{SGA1}).  The pro-finite set $\pi_1^{\text{\'et}}(X_{\bar k}; \bar x_1, \bar x_2)$ is a (left) torsor for $\pi_1^{\text{\'et}}(X_{\bar k}, \bar x_1)$; we let $\pi_1^\ell(X_{\bar k}; \bar x_1, \bar x_2)$ be the associated torsor for $\pi_1^\ell(X_{\bar k}, \bar x_1)$.  One may easily check that the (right) action of $\pi_1^{\text{\'et}}(X_{\bar k}, \bar x_2)$ on $\pi_1^{\text{\'et}}(X_{\bar k}; \bar x_1, \bar x_2)$ descends to an action of $\pi_1^{\ell}(X_{\bar k}; \bar x_2)$ on $\pi_1^{\ell}(X_{\bar k}; \bar x_1, \bar x_2)$, making it into a (right) torsor for this latter group as well.

\begin{defn}
\begin{enumerate}
\item Let $$\mathbb{Z}_\ell[[\pi_1^\ell(X_{\bar k}; \bar x_1, \bar x_2)]]=\varprojlim_{\pi_1^\ell(X_{\bar k}; \bar x_1, \bar x_2)\twoheadrightarrow H} \mathbb{Z}_\ell[H],$$ where the inverse limit is taken over all (continuous) finite quotients $H$ of $\pi_1^\ell(X_{\bar k}; \bar x_1, \bar x_2)$.  If $\bar x=\bar x_1 =\bar x_2$, this is the $\ell$-adic group ring $\mathbb{Z}_\ell[[\pi_1^\ell(X_{\bar k}, \bar x)]]$.  Note that $\mathbb{Z}_\ell[[\pi_1^\ell(X_{\bar k}; \bar x_1, \bar x_2)]]$ is naturally a $\mathbb{Z}_\ell[[\pi_1^\ell(X_{\bar k}, \bar x_1)]]$-$\mathbb{Z}_\ell[[\pi_1^\ell(X_{\bar k}, \bar x_2)]]$-bimodule.
\item  Let $$\mathscr{I}(\bar x)\subset \mathbb{Z}_\ell[[\pi_1^\ell(X_{\bar k}, \bar x)]]$$ be the augmentation ideal (the kernel of the augmentation map $$\mathbb{Z}_\ell[[\pi_1^\ell(X_{\bar k}, \bar x)]]\to \mathbb{Z}_\ell$$ sending $g\mapsto 1$ for $g\in \pi_1^\ell(X_{\bar k}, \bar x)$).  Let $$\mathscr{I}(\bar x_1, \bar x_2)^n=\mathscr{I}(\bar x_1)^n\cdot\mathbb{Z}_\ell[[\pi_1^\ell(X_{\bar k}; \bar x_1, \bar x_2)]]=\mathbb{Z}_\ell[[\pi_1^\ell(X_{\bar k}; \bar x_1, \bar x_2)]]\cdot \mathscr{I}(\bar x_2)^n.$$  We call this filtration the $\mathscr{I}$-adic filtration, and if the basepoints are clear, denote it by $\mathscr{I}^n$.  Note that the $\mathscr{I}$-adic topology on $\mathbb{Z}_\ell[[\pi_1^\ell(X_{\bar k}; \bar x_1, \bar x_2)]]$ is coarser than the profinite topology.
\item Let $$\mathbb{Q}_\ell[[\pi_1^\ell(X_{\bar k}; \bar x_1, \bar x_2)]]=\varprojlim_n \left(\mathbb{Z}_\ell[[\pi_1^\ell(X_{\bar k}; \bar x_1, \bar x_2)]]/\mathscr{I}(\bar x_1, \bar x_2)^n\otimes \mathbb{Q}_\ell\right),$$ topologized via the $\mathscr{I}$-adic topology.  If $\bar x=\bar x_1=\bar x_2$, this is $\mathbb{Q}_\ell[[\pi_1^\ell(X_{\bar k}, \bar x)]]$, the $\mathbb{Q}_\ell$-Mal'cev Hopf algebra of $\pi_1^\ell(X_{\bar k}, \bar x)$ (see e.g. \cite[Appendix A]{quillen-rational}).  We abuse notation and again denote the $\mathscr{I}$-adic filtration on $\mathbb{Q}_\ell[[\pi_1^\ell(X_{\bar k}; \bar x_1, \bar x_2)]]$ inherited from $\mathbb{Z}_\ell[[\pi_1^\ell(X_{\bar k}; \bar x_1, \bar x_2)]]$ by $\mathscr{I}^n$.
\end{enumerate}
\end{defn}
\begin{example}\label{noncommutative-power-series-example}
Suppose that $\pi_1^\ell(X_{\bar k}, \bar x)$ is a free pro-$\ell$ group on $m$ generators $\gamma_1, \cdots, \gamma_m$ (for example, if $X$ is an affine curve).  Then the map $\gamma_i\mapsto T_i+1$ induces isomorphisms $$\mathbb{Z}_\ell[[\pi_1^\ell(X_{\bar k}, \bar x)]]\overset{\sim}{\to} \mathbb{Z}_\ell\langle\langle T_1, \cdots, T_m\rangle\rangle,$$ $$\mathbb{Q}_\ell[[\pi_1^\ell(X_{\bar k}, \bar x)]]\overset{\sim}{\to} \mathbb{Q}_\ell\langle\langle T_1, \cdots, T_m\rangle\rangle,$$ where $R\langle\langle\cdots\rangle\rangle$ denotes the ring of noncommutative power series over $R$.  In both cases, $\mathscr{I}$ is the two-sided ideal generated by $T_1, \cdots, T_m$. (See e.g. \cite[$\S1$]{ihara3}.)
\end{example}

If $Y$ is a $k$-scheme, we define $$H_1(Y_{\bar k}, \mathbb{Z}_\ell):=\on{Hom}(H^1(Y_{\bar k}, \mathbb{Z}_\ell), \mathbb{Z}_\ell),$$ $$H_1(Y_{\bar k}, \mathbb{Q}_\ell):=\on{Hom}(H^1(Y_{\bar k}, \mathbb{Q}_\ell), \mathbb{Q}_\ell)$$ where $H^1(Y_{\bar k}, \mathbb{Z}_\ell)$ is $\ell$-adic cohomology, i.e.~the inverse limit of the $\mathbb{Z}/\ell^n\mathbb{Z}$-\'etale cohomology.     It is well-known that $$H_1(X_{\bar k}, \mathbb{Z}_\ell)\simeq\pi_1^{\ell}(X_{\bar k}, \bar x)^{\text{ab}}$$ canonically for any geometric point $\bar x$ of $X$.
\begin{prop}\label{abelianization-prop} Let $R=\mathbb{Z}_\ell$ or $\mathbb{Q}_\ell$, and let $\mathscr{I}(\bar x)$ be the augmentation ideal of $R[[\pi_1^\ell(X_{\bar k}, \bar x)]]$.  Then:
\begin{enumerate}
\item The map $g\mapsto g-1$ induces a (canonical) isomorphism $$H_1(X_{\bar k}, R)\simeq \pi_1^\ell(X_{\bar k}, \bar x)^{\on{ab}}\otimes_{\mathbb{Z}_\ell} R\overset{\sim}{\to} \mathscr{I}(\bar x)/\mathscr{I}(\bar x)^2.$$
\item Composition with any element of $\pi_1^\ell(X_{\bar k}; \bar x_1, \bar x_2)$ induces isomorphisms $$\mathscr{I}(\bar x_1)/\mathscr{I}(\bar x_1)^2\simeq \mathscr{I}(\bar x_1, \bar x_2)/\mathscr{I}(\bar x_1, \bar x_2)^2\simeq \mathscr{I}(\bar x_2)/\mathscr{I}(\bar x_2)^2.$$ These isomorphisms are independent of the choice of element.
\end{enumerate}
\end{prop}

\begin{proof}

\begin{enumerate}
\item This is a standard fact about pro-$\ell$ groups. 
\item The inverse morphism is given by composition with the inverse element, in $\pi_1^\ell(X_{\bar k}; \bar x_2, \bar x_1).$  Independence follows from the fact that if $p_1, p_2$ are two elements of $\pi_1^\ell(X_{\bar k}; \bar x_1, \bar x_2)$, then $p_1-p_2\in \mathscr{I}$, as $p_1-p_2=(1-p_2p_1^{-1})\cdot p_1$.  So if $x\in \mathscr{I}$, $x\cdot (p_1-p_2)\in \mathscr{I}^2$, and hence is zero in $\mathscr{I}/\mathscr{I}^2$.
\end{enumerate}
\end{proof}

\subsection{The weight filtration}
Let $\bar x$ be a geometric point of $X$, and $R=\mathbb{Z}_\ell$ or $R=\mathbb{Q}_\ell$. Let $$\mathscr{J}(\bar x)=\on{ker}(R[[\pi_1^\ell(X_{\bar k}, \bar x)]]\to R[[\pi_1^\ell(\overline{X}_{\bar k}, \bar x)]]),$$ where the morphism above is induced by the open embedding $X\hookrightarrow \overline{X}$.  We now define the weight filtration $W^\bullet$ on $R[[\pi_1^\ell(X_{\bar k}, \bar x)]]$.  This is an increasing, multiplicative filtration indexed by nonpositive integers.
\begin{defn} For $R=\mathbb{Q}_\ell$:
\begin{itemize}
\item $W^i=\mathbb{Q}_\ell[[\pi_1^\ell(X_{\bar k}, \bar x)]]$ for $i\geq 0$;
\item $W^{-1}=\mathscr{I}(\bar x)$; $W^{-2}=\mathscr{I}(\bar x)^2+\mathscr{J}$;
\item $W^{-i}=\sum_{a+b=i, a,b>0} W^{-a}\cdot W^{-b}$ for $i>2$.
\end{itemize}
For $R=\mathbb{Z}_\ell$, let $\iota: \mathbb{Z}_\ell[[\pi_1^\ell(X_{\bar k}, \bar x)]]\to \mathbb{Q}_\ell[[\pi_1^\ell(X_{\bar k}, \bar x)]]$ be the natural map, and set $W^i\mathbb{Z}_\ell[[\pi_1^\ell(X_{\bar k}, \bar x)]]=\iota^{-1}(W^i\mathbb{Q}_\ell[[\pi_1^\ell(X_{\bar k}, \bar x)]]).$
\end{defn}
\begin{remark}
One would obtain the same filtration by defining $$W^{-i}:= \mathscr{I}(\bar x)\cdot W^{-i+1}+\mathscr{J}(\bar x)\cdot W^{-i+2}$$ but the definition above will make several proofs easier; we will not need the equivalence between these two definitions.
\end{remark}
\begin{prop}\label{W-adic-topology-prop}
$\mathscr{I}^n\subset W^{-n}$ and $W^{-2n-1}\subset \mathscr{I}^n.$  
In particular, the $W^\bullet$-adic topology is the same as the $\mathscr{I}$-adic topology.
\end{prop}
\begin{proof}
 Both inclusions follow by induction on $n$; one uses that $\mathscr{I}^2\subset\mathscr{I}^2+\mathscr{J}\subset \mathscr{I}$.
\end{proof}
Now let $x\in X(k)$ be a rational point of $X$, and $\bar x$ the associated geometric point (given by our choice of algebraic closure of $k$), so that $G_k:=\on{Gal}(\bar k/k)$ acts naturally on $\pi_1^\ell(X_{\bar k}, \bar x)$.  The main theorem of this section is:
\begin{theorem}\label{quasi-scalar-pi1}
For all $\alpha\in \mathbb{Z}_\ell^\times$ sufficiently close to $1$, there exists $\sigma_\alpha\in G_k$ such that, for all $i$, $\sigma_\alpha$ acts on $\on{gr}^{-i}_{W}\mathbb{Q}_\ell[[\pi_1^\ell(X_{\bar k}, \bar x)]]$ via  $\alpha^i\cdot \on{Id}$.
\end{theorem}
Before giving the proof, we need several lemmas.

\begin{lemma}\label{H1-semisimple}
Let $F$ be a finite field, let $Y/F$ be a smooth, geometrically connected variety, and let $\overline{Y}$ be a smooth compactification of $Y$ with simple normal crossings boundary.  Then $G_F$ acts semi-simply on the $\ell$-adic cohomology group $H^1(Y_{\bar F}, \mathbb{Q}_\ell).$  Furthermore, this $G_F$-representation is mixed of weights $1$ and $2$, with the weight $1$ piece given by the image of the  natural map $$H^1(\overline{Y}_{\bar F}, \mathbb{Q}_\ell)\to H^1(Y_{\bar F}, \mathbb{Q}_\ell).$$
\end{lemma}
\begin{proof}
Let $j: Y\to \overline{Y}$ be the embedding.  Let $E_1, \cdots, E_n$ be the components of $\overline{Y}\setminus Y$.  Then the Leray spectral sequence for $Rj_*$ (see e.g.~\cite[6.2]{delignehodgei}) gives $$0\to H^1(\overline{Y}_{\bar F}, \mathbb{Q}_\ell)\to H^1(Y_{\bar F}, \mathbb{Q}_\ell)\to \bigoplus_{i\in \{1, \cdots, n\}} H^0(E_{i, \bar F}, \mathbb{Q}_\ell)\otimes \mathbb{Q}_\ell(-1)\to H^2(\overline{Y}_{\bar F}, \mathbb{Q}_\ell)\to \cdots$$
Now $H^1(\overline{Y}_{\bar F}, \mathbb{Q}_\ell)$ is pure of weight $1$ by the Weil conjectures, and $G_F$ acts on it semisimply \cite{tate2}.  On the other hand, let $$V=\ker\left(\bigoplus_{i\in \{1, \cdots, n\}} H^0(E_{i, \bar F}, \mathbb{Q}_\ell)\otimes \mathbb{Q}_\ell(-1)\to H^2(\overline{Y}_{\bar F}, \mathbb{Q}_\ell)\right).$$  $V$ is manifestly pure of weight $2$, with semisimple $G_F$-action.  But the short exact sequence $$0\to  H^1(\overline{Y}_{\bar F}, \mathbb{Q}_\ell)\to H^1(Y_{\bar F}, \mathbb{Q}_\ell)\to V\to 0$$ splits (canonically), as the first and last term have different weights.
\end{proof}

We now prove the following lemma, suggested to us by Will Sawin after reading an earlier draft of this paper --- it is closely analogous to results of Serre, Bogomolov, and Deligne (see e.g.~\cite[Corollaire 1]{bogomolov}).  It is an analogue of Theorem \ref{quasi-scalar-pi1}, but for $H_1(X_{\bar k}, \mathbb{Q}_\ell)$ rather than $\mathbb{Q}_\ell[[\pi_1^\ell(X_{\bar k}, \bar x)]]$.

\begin{lemma}\label{quasi-scalar-h1}
For all $\alpha\in \mathbb{Q}_\ell^\times$ sufficiently close to $1$, there exists $\sigma_\alpha\in G_k$ such that $\sigma_\alpha$ acts on $\on{gr}^{-i}_W H_1(X_{\bar k}, \mathbb{Q}_\ell)$ via multiplication by $\alpha^i$.  Here $W^{i}H_1(X_{\bar k}, \mathbb{Q}_\ell)=H_1(X_{\bar k}, \mathbb{Q}_\ell)$ for $i\geq -1$, $$W^{-2}H_1(X_{\bar k}, \mathbb{Q}_\ell)=\on{ker}(H_1(X_{\bar k}, \mathbb{Q}_\ell)\to H_1(\overline{X}_{\bar k}, \mathbb{Q}_\ell)),$$ and $W^iH_1(X_{\bar k}, \mathbb{Q}_\ell)=0$ for $i<-2$.
\end{lemma}
\begin{proof}
The proof proceeds in two steps.  First we show that there exists a $\sigma_\alpha$ as desired in the Zariski-closure of the image of the natural Galois representation $$\rho: G_k\to GL(H_1(X_{\bar k}, \mathbb{Q}_\ell)).$$ Second, we use the fact that the image of $\rho$ is open in its Zariski-closure to conclude.

\emph{Step 1.}  We first show that for \emph{any} $\alpha\in \mathbb{Q}_\ell^\times$, there exists an element $\gamma_\alpha$ of $\overline{\on{im}(\rho)}$ preserving $W^\bullet H_1(X_{\bar k}, \mathbb{Q}_\ell)$ and acting on $\on{gr}^{-i}_WH_1(X_{\bar k}, \mathbb{Q}_\ell)$ via $\alpha^i$.  Here $\overline{\on{im}(\rho)}$ is the Zariski-closure of the image of $\rho$.

The  filtration $W^\bullet$ in the statement of the lemma agrees with the weight filtration of Deligne (see e.g.~\cite[6]{delignehodgei}), so in particular there exist (many) Frobenii $\gamma$ in $G_k$ such that $\gamma$ acts with weight $i$ on $\on{gr}^{i}_WH_1(X_{\bar k}, \mathbb{Q}_\ell)$.  (Recall that this means that this means that if $\lambda$ is a generalized eigenvalues of $\gamma$ acting on $\on{gr}^i_W$, then $\lambda$ is an algebraic number, and $|\lambda|=q^{-i/2}$ for any embedding of $\mathbb{Q}(\lambda)$ into $\mathbb{C}$.) Choose such a $\gamma$.  

There is a short exact sequence $$0\to \on{gr}_W^{-2}H_1(X_{\bar k}, \mathbb{Q}_\ell)\to H_1(X_{\bar k}, \mathbb{Q}_\ell)\to \on{gr}_W^{-1}H_1(X_{\bar k}, \mathbb{Q}_\ell)\to 0.$$  Because $\gamma$ acts on $\on{gr}^{-i}_W$ with weight $-i$, there is a canonical $\gamma$-equivariant splitting of this sequence, so we have a canonical $\gamma$-equivariant isomorphism \begin{equation}\label{H1-splitting}\on{gr}_W^{-1}H_1(X_{\bar k}, L)\oplus \on{gr}_W^{-2}H_1(X_{\bar k}, L)\simeq H_1(X_{\bar k}, L).\end{equation}  

By Lemma \ref{H1-semisimple}, there exists a finite extension $L$ of $\mathbb{Q}_\ell$ so that $\gamma$ acts diagonalizably on $H_1(X_{\bar k}, L)$.  Choose a basis $\{e_i\}$ of $\gamma$-eigenvectors of $H_1(X_{\bar k}, L)$ adapted to splitting \ref{H1-splitting}, so that $e_1, \cdots, e_r$ forms a basis of $\on{gr}^{-2}_W$, and $e_{r+1}, \cdots, e_{m}$ forms a basis of $\on{gr}^{-1}_W.$ By the definition of weights, we have $$\gamma\cdot e_i=\lambda_i e_i,$$ with $|\lambda_i|=q$ if $1\leq i\leq r$ and $|\lambda_i|=q^{1/2}$ otherwise, where $|\cdot|$ is defined via any embedding of $\mathbb{Q}(\lambda_1, \cdots, \lambda_m)$ into $\mathbb{C}$.  Now the identity component $T$ of $\overline{\{\gamma^n\}_{n\in \mathbb{Z}}}$ is a subtorus of the diagonal torus $D$ of $GL(H_1(X_{\bar k}, L))$ (in the basis $\{e_i\}$).  

Let $X^*(D), X^*(T)$ be the character lattices of $D, T$ respectively; identify $X^*(D)\simeq \mathbb{Z}^m$ via the basis $\{e_i\}$.  The inclusion $T\hookrightarrow D$ induces a surjection $X^*(D)\twoheadrightarrow X^*(T)$, with kernel $K$ given by $\underline{a}\in \mathbb{Z}^m$ such that $$\prod_{i=1}^m \lambda_i^{a_i}=1.$$ $T$ is precisely the torus cut out by the characters in $K$, i.e.~the subtorus given by diagonal matrices $M$ such that $\chi(M)=1$ for all $\chi\in K$. But in particular, this holds for the matrices $$q^a\cdot\on{Id}_{\on{gr}^{-1}_W}~\oplus ~q^{2a} \cdot\on{Id}_{\on{gr}^{-2}_W},~ a\in\mathbb{Z}$$ by the multiplicativity of absolute values.  Hence these matrices are in $T$, which thus contains their Zariski-closure, the entire torus $T'$ $$\alpha\cdot\on{Id}_{\on{gr}^{-1}_W}~\oplus ~\alpha^{2} \cdot\on{Id}_{\on{gr}^{-2}_W},~ \alpha\in\mathbb{Q}_\ell^\times.$$  As the splitting \ref{H1-splitting} is defined over $\mathbb{Q}_\ell$, this torus is in fact in $\overline{\on{im}(\rho)}$, as desired. 

\emph{Step 2.}  Now we demonstrate that $\on{im}(\rho)$ is open in $\overline{\on{im}(\rho)}$, and use this fact to conclude the proof of the lemma.  In the case $k$ is a number field, this follows from the main result of \cite[Th\'{e}or\`{e}me 1]{bogomolov}, once we verify that for all places $\mathfrak{l}$ of $k$ lying over $\ell$, the restriction of $\rho$ to the decomposition group at $\mathfrak{l}$ is Hodge-Tate.  But in fact this representation is de Rham, hence Hodge-Tate, by the main result of \cite{kisin}.  For $k$ an arbitrary finitely-generated field of characteristic zero, we may reduce to the number field case by the argument of \cite[letter to Ribet of 1/1/1981, \S 1]{serre}.

To conclude, we observe that $\on{im}(\rho)\cap T'$ is thus open and non-empty, and in particular contains a non-empty open neighborhood of $1\in T'$.  This is exactly what we wanted to prove.
\end{proof}
\subsection{Semisimplicity}
The purpose of this subsection is to deduce Theorem \ref{quasi-scalar-pi1} from Lemma \ref{quasi-scalar-h1}.  Before doing so, we will need a result (which may be of independent interest) proving that certain elements of $G_k$ act semisimply on $\mathbb{Q}_\ell[[\pi_1^\ell(X_{\bar k}, \bar x)]]$.

\begin{theorem}\label{quasi-scalar-semisimple}
Let $\sigma_\alpha$ be as in Lemma \ref{quasi-scalar-h1}, with $\alpha$ not a root of unity; let $x$ be a rational point of $X$.  Then $\sigma_\alpha$ acts semisimply on $\mathbb{Q}_\ell[[\pi_1^\ell(X_{\bar k}, \bar x)]]/\mathscr{I}^n$ for all $n$.
\end{theorem}

The proof of this theorem is rather easy if $H_1(X_{\bar k}, \mathbb{Q}_\ell)$ is pure (in this case, the $\mathscr{I}$-adic filtration splits $\sigma_\alpha$-equivariantly for formal reasons), but requires some work in the mixed setting.  Before proceeding with the proof of Theorem \ref{quasi-scalar-semisimple}, we will need some auxiliary definitions and results.

Let $k((t^{1/\infty}))=\bigcup_n k((t^{1/n}))$.  Choose an algebraic closure $\overline{k((t^{1/\infty}))}$ and an identification of $\bar k$ with the algebraic closure of $k$ in $\overline{k((t^{1/\infty}))}$.
\begin{prop}\label{puiseux-prop}
The inclusion $k\hookrightarrow k((t^{1/\infty}))$ induces an isomorphism $G_k\simeq G_{k((t^{1/\infty}))}$. 
\end{prop}
\begin{proof}
This follows from the algebraic closedness of the field of Puiseux series over an algebraically closed field of characteristic zero.
\end{proof}
\begin{defn}
A \emph{rational tangential basepoint} of $X$ is a $k((t))$-point of $X$.  If $x$ is a rational tangential basepoint of $X$, we let $\bar x$ denote the geometric point of $X$ obtained from our choice of algebraic closure of $k((t^{1/\infty}))$ (which is also an algebraic closure of $k((t))$).
\end{defn}
Because of Proposition \ref{puiseux-prop}, a rational tangential basepoint $x$ of $X$ induces a natural action of $G_k$ on $\pi_1^{\ell}(X_{\bar k}, \bar x)$.
\begin{prop}\label{I-adic-weights}
Let $x_1, x_2$ be rational points or rational tangential basepoints of $X$, and $\bar x_1, \bar x_2$ the associated geometric points.  Then the $G_k$-representation on the $\mathbb{Q}_\ell$-vector space $$\mathscr{I}(\bar x_1, \bar x_2)^n/\mathscr{I}(\bar x_1, \bar x_2)^{n+1}$$ (where if one or both of the $x_i$ is a rational tangential basepoint, we view this as a $G_k$-representation via the isomorphism from Proposition \ref{puiseux-prop}) is mixed  with weights in $[-2n, -n]$.  For $\sigma_\alpha$ as in Lemma \ref{quasi-scalar-h1}, $\sigma_\alpha$ acts on $\mathscr{I}(\bar x_1, \bar x_2)^n/\mathscr{I}(\bar x_1, \bar x_2)^{n+1}$ semi-simply, with eigenvalues contained in $\{\alpha^n, \alpha^{n+1}, \cdots, \alpha^{2n}\}$.
\end{prop}
\begin{proof}
If $n=0$ this is trivial; if $n=1$ it follows from Propositions \ref{abelianization-prop} and \ref{H1-semisimple}.  In general, note that the composition map $$(\mathscr{I}(\bar x_1)/\mathscr{I}(\bar x_1)^2)^{\otimes n-1}\otimes \mathscr{I}(\bar x_1, \bar x_2)/\mathscr{I}(\bar x_1, \bar x_2)^2\to \mathscr{I}(\bar x_1, \bar x_2)^n/\mathscr{I}(\bar x_1, \bar x_2)^{n+1}$$ is Galois-equivariant and surjective.  As weights are additive in tensor products, this completes the proof.
\end{proof}
\begin{propconstr}\label{canonical-path}
Let $\sigma_\alpha$ be as in Lemma \ref{quasi-scalar-h1}, with $\alpha$ not a root of unity.  Let $x_1, x_2$ be rational points or rational tangential basepoints of $X$.  Then there exists a unique element $p(\bar x_1, \bar x_2)\in \mathbb{Q}_\ell[[\pi_1^\ell(X_{\bar k}; \bar x_1, \bar x_2)]]$ characterized by the following two properties:
\begin{enumerate}
\item $p(\bar x_1, \bar x_2)$ is fixed by $\sigma_\alpha$, and
\item $\epsilon(p(\bar x_1, \bar x_2))=1$, where $\epsilon:\mathbb{Q}_\ell[[\pi_1^\ell(X_{\bar k}; \bar x_1, \bar x_2)]]\to \mathbb{Q}_\ell$ is the augmentation map.
\end{enumerate}
We refer to $p(\bar x_1, \bar x_2)$ as the \emph{canonical path} between $\bar x_1, \bar x_2$.
\end{propconstr}
\begin{proof}
Such $p(\bar x_1, \bar x_2)$ are in bijection with $\sigma_\alpha$-equivariant splittings $s$ of $\epsilon$, by setting $p(\bar x_1, \bar x_2)=s(1)$.  So it is enough to prove that there is a unique $\sigma_\alpha$-equivariant splitting of $\epsilon$.

Because $\mathbb{Q}_\ell[[\pi_1^\ell(X_{\bar k}; \bar x_1, \bar x_2)]]$ is complete with respect to the $\mathscr{I}$-adic filtration, it is enough to show that the augmentation maps $$\epsilon_n: \mathbb{Q}_\ell[[\pi_1^\ell(X_{\bar k}; \bar x_1, \bar x_2)]]/\mathscr{I}^n\to \mathbb{Q}_\ell$$ admit unique $\sigma_\alpha$-equivariant splittings.  But $$\ker(\epsilon_n)=\mathscr{I}/\mathscr{I}^n,$$ on which $\sigma_\alpha$ acts with generalized eigenvalues in $\{\alpha, \cdots, \alpha^{2n}\}$ by Proposition \ref{I-adic-weights}, which does not contain $1$ by assumption.  So the result is clear.
\end{proof}
\begin{remark}\label{canonical-paths-rmk}
We have $p(\bar x, \bar x)=1$, and by uniqueness $p(\bar x_1, \bar x_2)\cdot p(\bar x_2, \bar x_3)=p(\bar x_1, \bar x_3)$.  Hence in particular $p(\bar x_1, \bar x_2)=p(\bar x_2, \bar x_1)^{-1}$.
\end{remark}
\begin{proof}[Proof of Theorem \ref{quasi-scalar-semisimple}]
Let $r_n$ be the quotient map $r_n: \mathscr{I}/\mathscr{I}^n\to\mathscr{I}/\mathscr{I}^2.$ First, we argue that it is enough to produce a $\sigma_\alpha$-equivariant splitting $s_n$ of the map $r_n$.  Indeed, then $s_n$ induces a surjective, $\sigma_\alpha$-equivariant map 
$$\xymatrix{
\bigoplus_{i=0}^{n-1} (\mathscr{I}/\mathscr{I}^2)^{\otimes i}\ar[rr]^{\oplus s_n^{\otimes i}}& &\mathbb{Q}_\ell[[\pi_1^{\ell}(X_{\bar k}, \bar x)]]/\mathscr{I}^n.
}$$  
As $\sigma_\alpha$ acts semi-simply on the source of this map, it does so on the target as well.

We now construct such an $s_n$. Let $V_1\oplus V_2$ denote the splitting of $\mathscr{I}/\mathscr{I}^2\simeq H_1(X_{\bar k}, \mathbb{Q}_\ell)$ into $\sigma_\alpha$-eigenspaces, where $\sigma_\alpha$ acts on $V_1$ via $\alpha\cdot \on{Id}$ and on $V_2$ via $\alpha^2\cdot \on{Id}$.  Let $t_i: \mathscr{I}/\mathscr{I}^2\to V_i, i=1,2$ be the natural quotient maps.  It suffices to construct splittings of $t_i\circ r_n$ for each $i$.  

For $i=1$, note that by Proposition \ref{I-adic-weights}, the generalized eigenvalues of $\sigma_\alpha$ on $\ker(t_1\circ r_n)$ are contained in $\{\alpha^2, \alpha^3, \cdots, \alpha^{2n-2}\}$; as $\alpha$ does not appear on this list, there is a (unique) $\sigma_\alpha$-equivariant splitting.  

So we need only construct a $\sigma_\alpha$-equivariant splitting of $t_2\circ r_n$.  In this case, there is no ``weight" reason for the existence of such a splitting: $\alpha^2$ may very well appear as an eigenvalue of $\sigma_\alpha$ acting on $\ker(t_2\circ r_n)$.  Instead, we give a direct construction.

Without loss of generality, we may assume that $V_2$ is nonzero. We may replace $k$ by a finite extension so that each of the $c$ components $D_i$ of the simple normal crossings divisor $D=\bigcup_i D_i=\overline{X}\setminus X$ is geometrically connected and has a rational point $x_i^0$ in the smooth locus of $D$.  (We replace $\sigma_\alpha$ by a power so it lies in the Galois group of this extension of $k$; it suffices to prove semisimplicity for this power.) For each $i$, choose a $k[[t]]$-point of $\overline{X}$ transverse to $D_i$, so that the special point of $\on{Spec}(k[[t]])$ is sent to $x_i^0$.  Let $x_i$ be the associated $k((t))$-point, $\tilde x_i$ the associated $\bar k((t))$-point, and $\bar x_i$ the associated $\overline{k((t))}$-point. 

Note that the map $\sqcup_i \tilde x_i\to X$ induces a surjection $$u: \mathbb{Q}_\ell(1)^{\oplus c}\simeq \oplus_i H_1(\tilde x_i, \mathbb{Q}_\ell)\to H_1(X_{\bar k}, \mathbb{Q}_\ell)\twoheadrightarrow V_2,$$ by e.g. the proof of Lemma \ref{H1-semisimple}. We first argue that it is enough to produce a $\sigma_\alpha$-equivariant lift of this map to $\mathscr{I}/\mathscr{I}^n$, i.e.~to construct the dotted arrow $w$ below.
$$\xymatrix{
 & & \mathbb{Q}_\ell(1)^{\oplus c}\simeq \oplus_i H_1(\tilde x_i, \mathbb{Q}_\ell) \ar@{>>}[d]^u \ar@/_15pt/@{.>}[lld]_w\\
\mathscr{I}/\mathscr{I}^n\ar@{>>}[r]^-{r_n} & \mathscr{I}/\mathscr{I}^2=H_1(X_{\bar k}, \mathbb{Q}_\ell) \ar@{>>}[r]^-{t_2} &V_2.
}$$
Indeed, $\sigma_\alpha$ evidently acts semisimply on $\mathbb{Q}_\ell(1)^{\oplus c}$, so the vertical map $u$ above admits a $\sigma_\alpha$-equivariant section $s$.  Composing it with the dotted arrow $w$ will give the desired section to $t_2\circ r_n$.

We now produce $w$ as desired.  Let $\gamma$ be a generator of the maximal pro-$\ell$ quotient $I$ of the inertia group of $G_{k((t))}$; the map $x_i\to X$ induces a map $$\iota_{x_i}: I\to \pi_1^\ell(X_{\bar k}, \bar x_i).$$
\begin{figure}[h]
\centering{
\hspace*{1cm} \resizebox{130mm}{!}{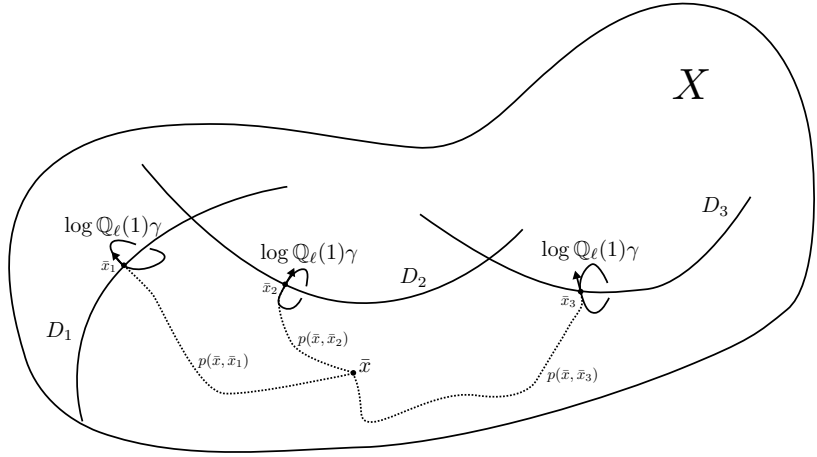}
\caption{A schematic depiction of the lift of the weight $-2$ part of $\mathscr{I}(\bar x)/\mathscr{I}(\bar x)^2$ to $\mathbb{Q}_\ell[[\pi_1^{\ell}(X_{\bar k}, \bar x)]]$}\label{weight2-schematic}
}
\end{figure} 
Observe that there is a canonical isomorphism $I\simeq H_1(\tilde x_i, \mathbb{Z}_\ell)\simeq \mathbb{Z}_\ell(1)$. Note moreover that the maps $u_i: H_1(\tilde x_i, \mathbb{Z}_\ell)\to V_2$ (the direct summands of $u$) factor as $$u_i: H_1(\tilde x_i, \mathbb{Z}_\ell)\simeq I\overset{\iota_{x_i}}{\longrightarrow} \pi_1^{\ell}(X_{\bar k}, \bar x_i)\to \pi_1^{\ell}(X_{\bar k}, \bar x_i)^{\text{ab}}\otimes \mathbb{Q}_\ell\simeq H_1(X_{\bar k}, \mathbb{Q}_\ell)\twoheadrightarrow V_2.$$  Of course it suffices to construct $\sigma_\alpha$-equivariant lifts of $u_i$ to $\mathscr{I}/\mathscr{I}^n$ for each $i$; we claim that $$\gamma\mapsto p(\bar x, \bar x_i)\cdot\log(\iota_{x_i}(\gamma))\cdot p(\bar x_i, \bar x)$$ 
is such a lift, where $p(\bar x, \bar x_i)$ is the canonical path from Proposition-Construction \ref{canonical-path}.  Here $\log(\iota_{x_i}(\gamma))$ is the power series $$\log(\iota_{x_i}(\gamma))=\log(1+(\iota_{x_i}(\gamma)-1))=\sum_{j=1}^\infty(-1)^{j+1}\frac{(\iota_{x_i}(\gamma)-1)^j}{j},$$ which is in fact a finite sum because $(\iota_{x_i}(\gamma)-1)\in \mathscr{I}$.

We must check that this is (a) a lift, and (b) $\sigma_\alpha$-equivariant.    To check statement (a), we must show that $$p(\bar x, \bar x_i)\cdot\log(\iota_{x_i}(\gamma))\cdot p(\bar x_i, \bar x)=\iota_{x_i}(\gamma)-1 \bmod \mathscr{I}^2,$$ by Proposition \ref{abelianization-prop}(1).  Because $(\iota_{x_i}(\gamma)-1)\in \mathscr{I}$, we have
\begin{align*}
p(\bar x, \bar x_i)\cdot\log(\iota_{x_i}(\gamma))\cdot p(\bar x_i, \bar x) &= p(\bar x, \bar x_i)\cdot(\iota_{x_i}(\gamma)-1)\cdot p(\bar x_i, \bar x)\bmod \mathscr{I}^2\\
&= \iota_{x_i}(\gamma)-1 \bmod \mathscr{I}^2,
\end{align*}
where the first equality follows from the power series definition of $\log$ and the second from the same argument as in the proof of Proposition \ref{abelianization-prop}(2).  To check statement (b), we must verify that $$\sigma_\alpha \left(p(\bar x, \bar x_i)\cdot\log(\iota_{x_i}(\gamma))\cdot p(\bar x_i, \bar x)\right)=\alpha^2\cdot p(\bar x, \bar x_i)\cdot\log(\iota_{x_i}(\gamma))\cdot p(\bar x_i, \bar x).$$
As $V_2$ was assumed to be non-zero, and is a quotient of $\mathbb{Q}_\ell(1)^{\oplus c}$, we know that $\alpha^2=\chi(\sigma_\alpha)$ where $\chi: G_k\to \mathbb{Z}_\ell^\times$ is the cyclotomic character.  Since $G_k$ acts on $I$ via the cyclotomic character, we have
\begin{align*}
\sigma_\alpha \left(p(\bar x, \bar x_i)\cdot\log(\iota_{x_i}(\gamma))\cdot p(\bar x_i, \bar x)\right) &= \sigma_\alpha(p(\bar x, \bar x_i))\cdot\log(\iota_{x_i}(\gamma^{\chi(\sigma_\alpha)}))\cdot \sigma_\alpha(p(\bar x_i, \bar x))\\
&=p(\bar x, \bar x_i)\cdot\log(\iota_{x_i}(\gamma)^{\alpha^2})\cdot p(\bar x_i, \bar x)\\
&=\alpha^2\cdot p(\bar x, \bar x_i)\cdot\log(\iota_{x_i}(\gamma))\cdot p(\bar x_i, \bar x)
\end{align*}
as desired, where the first equality follows from the fact that composition of paths and $\iota_{x_i}$ commute with the $G_k$-action, the second from the fact that the $p(\bar x, \bar x_i)$ are $\sigma_\alpha$-invariant (by definition), and the last from the identity $\log(y^n)=n\log(y)$.  This completes the proof.
\end{proof}
\begin{remark}
An essentially identical argument proves:
\begin{theorem}\label{frobenius-semisimple}
Let $Y/\mathbb{F}_q$ be a smooth, geometrically connected variety admitting a simple normal crossings compactification, with $\ell$ prime to $q$.  Let $y\in Y(\mathbb{F}_q)$ be a rational point.  Then Frobenius acts semisimply on $\mathbb{Q}_\ell[[\pi_1^\ell(Y_{\overline{\mathbb{F}_q}}, \bar y)]]/\mathscr{I}^n$ for any $n$.
\end{theorem}
The only difference in the proof is that one replaces $\alpha^j$ with a $q$-Weil number of weight $-j$ whenever necessary.
\end{remark}
\begin{proof}[Proof of Theorem \ref{quasi-scalar-pi1}]
It suffices to prove the theorem with $\mathbb{Q}_\ell[[\pi_1^\ell(X_{\bar k}, \bar x)]]$ replaced by $\mathbb{Q}_\ell[[\pi_1^\ell(X_{\bar k}, \bar x)]]/\mathscr{I}^n$ with the induced $W$-adic filtration, by Proposition \ref{W-adic-topology-prop}. 

Let $\sigma_\alpha$ be as in Lemma \ref{quasi-scalar-h1}; we claim that this same $\sigma_\alpha$ also satisfies the conclusions of Theorem \ref{quasi-scalar-pi1}.  Indeed, by Theorem \ref{quasi-scalar-semisimple}, we may choose a $\sigma_\alpha$-equivariant splitting $s$ of the quotient map $\mathscr{I}/\mathscr{I}^n\to \mathscr{I}/\mathscr{I}^2.$  The induced map 
$$\xymatrix{
\bigoplus_{i=0}^{n-1} H_1(X_{\bar k}, \mathbb{Q}_\ell)^{\otimes i}\simeq \bigoplus_{i=0}^{n-1} (\mathscr{I}/\mathscr{I}^2)^{\otimes i} \ar[rr]^-{\oplus s^{\otimes i}} && \mathbb{Q}_\ell[[\pi_1^\ell(X_{\bar k}, \bar x)]]/\mathscr{I}^n
}$$
is surjective and $\sigma_\alpha$-equivariant, so we are done by the multiplicativity of the $W^\bullet$-filtration.
\end{proof}


\section{Integral $\ell$-adic periods and convergent group rings}\label{integral-l-adic-periods-section}

\subsection{Convergent group rings} In the previous section we constructed (Theorem \ref{quasi-scalar-pi1}) certain elements $\sigma_{\alpha}\in G_k$ acting semisimply on $\mathbb{Q}_\ell[[\pi_1^\ell(X_{\bar k|},\bar x)]]$; in particular, $\mathbb{Q}_\ell[[\pi_1^\ell(X_{\bar k}, \bar x)]]$ admits a set of $\sigma_\alpha$-eigenvectors with dense span (in the $\mathscr{I}$-adic topology).  This is typically not the case for $\mathbb{Z}_\ell[[\pi_1^\ell(X_{\bar k}, \bar x)]]$.
\begin{example}
Suppose $X=\mathbb{G}_m$, and let $x$ be any $k$-rational point of $X$.  Then $\pi_1^\ell(X_{\bar k}, \bar x)=\mathbb{Z}_\ell(1)$.  Let $\gamma$ be a topological generator of $\mathbb{Z}_\ell(1)$.  Via the map $\gamma\mapsto T+1$ we have:
\begin{itemize}
\item $\mathbb{Z}_\ell[[\pi_1^\ell(X_{\bar k}, \bar x)]]\overset{\sim}{\to} \mathbb{Z}_\ell[[T]]$, $\mathscr{I}^n=W^{-2n}=W^{-2n+1}=(T^n)$;
\item $\mathbb{Q}_\ell[[\pi_1^\ell(X_{\bar k}, \bar x)]]\overset{\sim}{\to} \mathbb{Q}_\ell[[T]]$, $\mathscr{I}^n=W^{-2n}=W^{-2n+1}=(T^n)$.
\end{itemize}
For $\sigma \in G_k$, we have $$\sigma(1+T)=(1+T)^{\chi(\sigma)}$$ where $\chi: G_k\to \mathbb{Z}_\ell^\times$ is the cyclotomic character.  The elements $$(\log(1+T))^n\in \mathbb{Q}_\ell[[T]], n \in \mathbb{Z}_{\geq 0}$$ are $\sigma$-eigenvectors with eigenvalue $\chi(\sigma)^n$; their span is evidently dense in the $(T)$-adic topology, as $(\log(1+T))^n$ has leading term $T^n$.  On the other hand, if $\chi(\sigma)\not=1$, the only $\sigma$-eigenvector in $\mathbb{Z}[[T]]$ is $1$.
\end{example}
Observe that the $\sigma$-eigenvectors in the example above, $(\log(1+T))^n$, are power series with positive $\ell$-adic radius of convergence.  The purpose of this section is to generalize this observation to arbitrary $X$.
\begin{defn}
Let $r>0$ be a positive real number.  Let $$\pi_n: \mathbb{Q}_\ell[[\pi_1^\ell(X_{\bar k}, \bar x)]]\to \mathbb{Q}_\ell[[\pi_1^{\ell}(X_{\bar k}, \bar x)]]/\mathscr{I}^n$$ be the quotient map, and $v_n: \mathbb{Q}_\ell[[\pi_1^{\ell}(X_{\bar k}, \bar x)]]/\mathscr{I}^n\to \mathbb{Z}\cup \{\infty\}$ the valuation defined by $$v_n(p)=-\inf\{m\in \mathbb{Z}\mid \ell^m\cdot p\in \on{im}(\mathbb{Z}_\ell[[\pi_1^\ell(X_{\bar k}, \bar x)]]/\mathscr{I}^n\to \mathbb{Q}_\ell[[\pi_1^\ell(X_{\bar k}, \bar x)]]/\mathscr{I}^n)\}.$$
We define the \emph{convergent group ring of radius $\ell^{-r}$} to be $$\mathbb{Q}_\ell[[\pi_1^\ell(X_{\bar k}, \bar x)]]^{\leq \ell^{-r}}:=\{p\in \mathbb{Q}_\ell[[\pi_1^\ell(X_{\bar k}, \bar x)]]\mid v_n(\pi_n(p))+nr\to \infty \text{ as } n\to \infty\},$$ topologized via the $r$-Gauss norm $$|p|_r:=\sup_n\{\ell^{-v_n(\pi_n(p))-nr}\}.$$
We again use the notation  $\mathscr{I}^n, W^{-i}$ to denote the filtrations on $\mathbb{Q}_\ell[[\pi_1^\ell(X_{\bar k}, \bar x)]]^{\leq \ell^{-r}}$ inherited from $\mathbb{Q}_\ell[[\pi_1^\ell(X_{\bar k}, \bar x)]]$.
\end{defn}
Loosely speaking, $\mathbb{Q}_\ell[[\pi_1^\ell(X_{\bar k}, \bar x)]]^{\leq \ell^{-r}}$ consists of those elements whose denominators do not ``grow too quickly" modulo $\mathscr{I}^n$.  Note that it is \emph{not} given the subspace topology, but rather the Gauss norm topology, in all that follows.
\begin{example}\label{convergence-example}
Suppose $\pi_1^\ell(X_{\bar k}, \bar x)$ is a free pro-$\ell$ group, generated by $\gamma_1, \cdots, \gamma_m$ (for example, if $X$ is an affine curve). Then the map $\gamma_i\mapsto 1+T_i$ induces an isomorphism $$\mathbb{Q}_\ell[[\pi_1^\ell(X_{\bar k}, \bar x)]]\overset{\sim}{\to} \mathbb{Q}_\ell\langle\langle T_1, \cdots,T_m\rangle \rangle,$$ where $\mathbb{Q}_\ell\langle\langle T_1, \cdots,T_m\rangle \rangle$ is the ring of non-commutative power series in $T_1, \cdots, T_m$ over $\mathbb{Q}_\ell$. The subring $\mathbb{Q}_\ell[[\pi_1^\ell(X_{\bar k}, \bar x)]]^{\leq \ell^{-r}}$ consists of those power series $\sum a_I T^{I}$ such that $$\lim_{|I|\to \infty} v_\ell(a_I)+|I|r=\infty.$$  If $m=1$, this is precisely the set of univariate power series converging on the closed ball of radius $\ell^{-r}$.  The Gauss norm is given by $$\left|\sum a_IT^I\right|_r=\sup_I \ell^{-v_\ell(a_I)-|I|r}.$$  In fact $\mathbb{Q}_\ell[[\pi_1^\ell(X_{\bar k}, \bar x)]]^{\leq \ell^{-r}}$ is the completion of $\mathbb{Z}_\ell[[\pi_1^\ell(X_{\bar k}, \bar x)]]\otimes \mathbb{Q}_\ell$ at the norm $|\cdot |_r$.
\end{example}
The following proposition justifies the terminology \emph{convergent group ring}.
\begin{prop}\label{convergence-prop}
Suppose $\pi_1^\ell(X_{\bar k}, \bar x)$ is a finitely-generated free pro-$\ell$ group, and let $\rho: \pi_1^\ell(X_{\bar k}, \bar x)\to GL_n(\mathbb{Z}_\ell)$ be a continuous representation which is trivial mod $\ell^m$.  Then for any $0<r<m$, there exists a unique continuous ring homomorphism $\tilde \rho: \mathbb{Q}_\ell[[\pi_1^\ell(X_{\bar k}, \bar x)]]^{\leq \ell^{-r}}\to \mathfrak{gl}_n(\mathbb{Q}_\ell)$ making the diagram
$$\xymatrix{
\mathbb{Z}_\ell[[\pi_1^\ell(X_{\bar k}, \bar x)]] \ar[r]^-{\rho} \ar[d] & \mathfrak{gl}_n(\mathbb{Z}_\ell)\ar[d]\\
\mathbb{Q}_\ell[[\pi_1^\ell(X_{\bar k}, \bar x)]]^{\leq \ell^{-r}}\ar[r]^-{\tilde \rho} &\mathfrak{gl}_n(\mathbb{Q}_\ell)
}$$
commute, where the top horizontal arrow is the natural ring homomorphism induced by $\rho$, and the vertical arrows are the natural inclusions.
\end{prop}
\begin{proof}
Uniqueness is clear from the density of $\mathbb{Z}_\ell[[\pi_1^\ell(X_{\bar k}, \bar x)]]\otimes \mathbb{Q}_\ell$ in $\mathbb{Q}_\ell[[\pi_1^\ell(X_{\bar k}, \bar x)]]^{\leq \ell^{-r}}$, so it suffices to construct $\tilde \rho$.

If $M\in \mathfrak{gl}_n(\mathbb{Q}_\ell)$, let $$|M|:=\ell^{\inf\{r\in \mathbb{Z}\mid \ell^r\cdot M\in \mathfrak{gl}_n(\mathbb{Z}_\ell)\}}.$$ Using the notation of Example \ref{convergence-example}, we choose topological generators $\lambda_1, \cdots, \lambda_m$ of $\pi_1^\ell(X_{\bar k}, \bar x)$ and set $T_i=1+\lambda_i\in \mathbb{Q}_\ell[[\pi_1^\ell(X_{\bar k}, \bar x)]]^{\leq \ell^{-r}}$.  Then we set $$\tilde \rho\left(\sum_I a_I T^I\right)=\sum_I a_I \rho(T^I).$$  This sum converges because $|\rho(T_i)|=|\rho(\gamma_i)-\on{Id}|\leq \ell^{-m}$, as $\rho$ is trivial mod $\ell^m$, so $|\rho(T^I)|\leq \ell^{-m|I|}$.  But we have $$|a_I\rho(T^I)|\leq |a_I|\ell^{-m|I|}$$ which tends to zero because $|a_I|< \ell^{r|I|}$ for $|I|$ large, with $r<m$.

To check continuity, we must check that $\tilde\rho$ is a bounded operator, i.e.~ that there exists $C>0$ such that $$\left|\tilde \rho\left(\sum_I a_I T^I\right)\right|\leq C\left|\sum_I a_I T^I\right|_r.$$   But we have 
$$
\left|\tilde \rho\left(\sum_I a_I T^I\right)\right| \leq \sup_I |a_I|_\ell \cdot \left|\tilde\rho(T^I)\right|
\leq \sup_I \ell^{-v_\ell(a_I)-m|I|}
\leq\sup_I\ell^{-v_\ell(a_I)-r|I|} 
= \left|\sum_I a_I T^I\right|_r
$$
so we may take $C=1$.
\end{proof}
\begin{remark}
The previous proposition is true in much greater generality (i.e. it holds if $X$ satisfies Condition \ref{condition-star} below), but this case admits a simple proof and is all we require, as we may reduce our main theorems to the case where $X$ is an affine curve.
\end{remark}
\subsection{Integral $\ell$-adic periods}\label{integral-l-adic-periods}
Theorems \ref{quasi-scalar-semisimple} and \ref{quasi-scalar-pi1} imply that $\mathbb{Q}_\ell[[\pi_1^\ell(X_{\bar k}, \bar x)]]$ admits a set of $\sigma_\alpha$-eigenvectors with dense span (in the $\mathscr{I}$-adic topology).  The purpose of this section is to deduce an analogous statement for $\mathbb{Q}_\ell[[\pi_1^\ell(X_{\bar k}, \bar x)]]^{\leq \ell^{-r}}$ for $r\gg 0$.

Consider the following condition on $\mathbb{Z}_\ell[[\pi_1^{\ell}(X_{\bar k}, \bar x)]]$: 
\begin{equation}\label{condition-star}\tag{$\star$}
\mathscr{I}^n/\mathscr{I}^{n+1}\text{ is $\mathbb{Z}_\ell$-torsion-free for all $n$.}
\end{equation}
This condition ensures that the natural map $\mathbb{Z}_\ell[[\pi_1^{\ell}(X_{\bar k}, \bar x)]]\to \mathbb{Q}_\ell[[\pi_1^{\ell}(X_{\bar k}, \bar x)]]$ is injective, and hence that $\on{gr}^{-i}_W\mathbb{Z}_\ell[[\pi_1^{\ell}(X_{\bar k}, \bar x)]]$ is $\mathbb{Z}_\ell$-torsion-free for all $i$. Condition \ref{condition-star} is satisfied if e.g.~$X$ is a smooth curve; the case of an affine curve (which is all we require) follows immediately from the isomorphism of Example \ref{noncommutative-power-series-example}.
\begin{theorem}\label{quasi-scalar-eigenvectors}
Let $\sigma_\alpha$ be as in Theorem \ref{quasi-scalar-pi1}, with $\alpha$ not a root of unity.  Suppose $X$ satisfies Condition \ref{condition-star}. Then there exits $r_\alpha>0$ such that for $r>r_\alpha$, $\mathbb{Q}_\ell[[\pi_1^\ell(X_{\bar k}, \bar x)]]^{\leq \ell^{-r}}$ contains a set of $\sigma_\alpha$-eigenvectors with dense span.  Moreover $W^{-n}\subset \mathbb{Q}_\ell[[\pi_1^\ell(X_{\bar k}, \bar x)]]^{\leq \ell^{-r}}$ admits a set of $\sigma_\alpha$-eigenvectors with eigenvalues in $\{\alpha^n, \alpha^{n+1}, \cdots\}$ and with dense span (in the topology defined by the $r$-Gauss norm).
\end{theorem}
Equivalently, if $y$ is a $\sigma_\alpha$-eigenvector in $\mathbb{Q}_\ell[[\pi_1^\ell(X_{\bar k}, \bar x)]]$, $-v_n(\pi_n(y))$ grows at most linearly in $n$.  We require several lemmas before giving the proof.
\begin{lemma}\label{l-adic-periods-bound}
Consider $\mathbb{Z}_\ell[[\pi_1^\ell(X_{\bar k}, \bar x)]]$ as a $\mathbb{Z}_\ell[\sigma_\alpha]$-module, with the $W^\bullet$-filtration.    Suppose $X$ satisfies Condition \ref{condition-star}.  Then for any $m\geq i+1$, $$\on{Ext}^1_{\mathbb{Z}_\ell[\sigma_\alpha]}(W^{-i}/W^{-i-1}, W^{-i-1}/W^{-m})$$ is annihilated by $\ell^{v(i,m,\alpha)}$, where $$v(i,m,\alpha)={\sum_{s=1}^{m-i-1} v_\ell(\alpha^s-1)}.$$
\end{lemma}
\begin{proof}
Note that Condition \ref{condition-star} implies that the natural map $\mathbb{Z}_\ell[[\pi_1^\ell(X_{\bar k}, \bar x)]]\to \mathbb{Q}_\ell[[\pi_1^\ell(X_{\bar k}, \bar x)]]$ is injective; hence $W^{-s}/W^{-s+1}$ is $\mathbb{Z}_\ell$-torsion-free for any $s$, and $\sigma_\alpha$ acts on it via the scalar $\alpha^s$.

We proceed by induction on $m$; the case $m=i+1$ is trivial.  For the induction step, from the short exact sequence $$0\to W^{-m+1}/W^{-m}\to W^{-i-1}/W^{-m}\to W^{-i-1}/W^{-m+1}\to 0$$ it suffices to show that $\on{Ext}^1_{\mathbb{Z}_\ell[\sigma_\alpha]}(W^{-i}/W^{-i-1}, W^{-m+1}/W^{-m})$ is annihilated by $\ell^{v_\ell(\alpha^{m-i-1}-1)}$.  But $\sigma_\alpha$ acts via the scalar $\alpha^i$ on $W^{-i}/W^{-i-1}$ and by $\alpha^{m-1}$ on $W^{-m+1}/W^{-m}$; thus it is enough to show that $$\on{Ext}^1_{\mathbb{Z}_\ell[\sigma_\alpha]}(\mathbb{Z}_\ell[\sigma_\alpha]/(\sigma_\alpha-\alpha^i), \mathbb{Z}_\ell[\sigma_\alpha]/(\sigma_\alpha-\alpha^{m-1}))$$ is annihilated by $\ell^{v_\ell(\alpha^{m-i-1}-1)}$.  But this is clear from the free resolution 
$$\xymatrix{
0\ar[r]& \mathbb{Z}_\ell[\sigma_\alpha]\ar[rr]^{\cdot(\sigma_\alpha-\alpha^i)} & & \mathbb{Z}_\ell[\sigma_\alpha] \ar[r] & \mathbb{Z}_\ell[\sigma_\alpha]/(\sigma_\alpha-\alpha^i)\ar[r] & 0.
}$$
\end{proof}
\begin{remark}
Note that the proof of this lemma required the semi-simplicity of the $\sigma_\alpha$-action on $\mathbb{Q}_\ell[[\pi_1^\ell(X_{\bar k}, \bar x)]]$ (i.e.~the full strength of Theorem \ref{quasi-scalar-pi1} --- Lemma \ref{quasi-scalar-h1} does not suffice).
\end{remark}
\begin{remark}
There are canonical classes $e_{i,m}(\sigma_\alpha, x)\in\on{Ext}^1_{\mathbb{Z}_\ell[\sigma_\alpha]}(W^{-i}/W^{-i-1}, W^{-i-1}/W^{-m})$ corresponding to the extensions $$0\to W^{-i-1}/W^{-m}\to W^{-i}/W^{-m}\to W^{-i}/W^{-i-1}\to 0.$$ We refer to the least $b$ such that $\ell^b$ annihilates $e_{i,m}(\sigma_\alpha, x)$ as an \emph{integral $\ell$-adic period} of $\pi_1^\ell(X_{\bar k}, \bar x)$, because these values measure the failure of the canonical $\sigma_\alpha$-equivariant isomorphism $$\mathbb{Q}_\ell[[\pi_1^\ell(X_{\bar k}, \bar x)]]/W^{-m}\mathbb{Q}_\ell[[\pi_1^\ell(X_{\bar k}, \bar x)]]\overset{\sim}{\to} \bigoplus_{i=0}^{m-1} \on{gr}_W^{-i}\mathbb{Q}_\ell[[\pi_1^\ell(X_{\bar k}, \bar x)]]$$ to preserve the integral structures on both sides. 
\end{remark}

\begin{lemma}\label{q-powers-bound}
Let $\ell$ be a prime and $q\in (1+\ell\mathbb{Z}_\ell)^{\times}$ if $\ell\not=2$, $q\in (1+4\mathbb{Z}_2)^{\times}$ if $\ell=2$.  Then 
$$v_\ell(q^k-1)=v_\ell(q-1)+v_\ell(k).$$
\end{lemma}
\begin{proof}
This follows from the fact that $(1+\ell\mathbb{Z}_\ell)^{\times}$ (resp.~$(1+4\mathbb{Z}_2)^{\times}$) is pro-cyclic with $(1+\ell\mathbb{Z}_\ell)^k=1+k\ell\mathbb{Z}_\ell$ (resp.~$(1+4\mathbb{Z}_2)^k=1+4k\mathbb{Z}_\ell$).
\end{proof}
\begin{lemma}\label{better-q-power-bound}
Let $\ell$ be a prime and $q\in \mathbb{Z}_\ell^\times$.  Let $r$ be the order of $q$ in $\mathbb{F}_\ell^*$ if $\ell\not=2$ and in $(\mathbb{Z}/4\mathbb{Z}^*)$ if $\ell=2$.  Then $$v_\ell(q^k-1)=\begin{cases} v_\ell(q^r-1)+v_\ell(k/r) \text{ if $r\mid k$}\\ 0 \text{ if $r\nmid k$ and $\ell\not=2$}\\ 1 \text{ if $r\nmid k$ and $\ell=2$}.\end{cases}$$ 
\end{lemma}
\begin{proof}
As in the statement, we consider the three cases $r\mid k,$ $r\nmid k \text{ and }\ell\not=2,$ and $r\nmid k \text{ and }\ell=2.$  The first case is immediate from Lemma \ref{q-powers-bound}, replacing $q$ with $q^r$ and $k$ with $k/r$.  The other two cases follow from the definition of $r$.
\end{proof}
\begin{lemma}\label{q-power-sum-bound}
Let $q, \ell, r$ be as in Corollary \ref{better-q-power-bound}.  Let $$C(q, \ell, k)=\frac{k}{r}\left(v_\ell(q^r-1)+\frac{1}{\ell-1}\right)$$ if $\ell\not=2$ and $$C(q, \ell, k)= \frac{k}{r} \left( v_\ell(q^r-1)+\frac{1}{\ell-1}+1\right)+\frac{1}{r}$$ if $\ell=2$.  Then $$\sum_{i=1}^k v_\ell(q^i-1)\leq C(q, \ell, k).$$
\end{lemma}
\begin{proof}
In the case $\ell\not=2$, we have by Lemma \ref{better-q-power-bound} that $v_\ell(q^i-1)=v_\ell(q^r-1)+v_\ell(i/r)$ if $r\mid k$ and $0$ otherwise.  Thus we have $$\sum_{i=0}^k v_\ell(q^i-1)=\sum_{i=1}^{\lfloor k/r \rfloor} v_\ell(q^{ri}-1)=\sum_{i=1}^{\lfloor k/r \rfloor} (v_\ell(q^r-1)+v_\ell(i)).$$  But this last satisfies
\begin{align*}
\sum_{i=1}^{\lfloor k/r \rfloor} (v_\ell(q^r-1)+v_\ell(i)) &=\lfloor k/r\rfloor v_\ell(q^r-1)+\lfloor \frac{k}{r\ell}\rfloor+\lfloor \frac{k}{r\ell^2}\rfloor+\lfloor \frac{k}{r\ell^3}\rfloor+\cdots\\
&\leq \frac{k}{r} v_\ell(q^r-1)+\frac{k}{r\ell(1-\frac{1}{\ell})}\\
&=\frac{k}{r}\left(v_\ell(q^r-1)+\frac{1}{\ell-1}\right)
\end{align*}
In the case $\ell=2$, and $r=1$, an identical argument shows that $$\sum_{i=0}^k v_\ell(q^i-1)\leq \frac{k}{r}\left(v_\ell(q^r-1)+\frac{1}{\ell-1}\right).$$  If $\ell=2$ and $r=2$, we have 
\begin{align*}
\sum_{i=1}^k v_\ell(q^i-1) &= \sum_{i \text{ odd}, 1\leq i\leq k} 1 +\sum_{i=1}^{\lfloor k/2\rfloor} (v_\ell(q^r-1)+v_\ell(i))\\
&\leq \frac{k+1}{r} +\frac{k}{r}v_\ell(q^r-1)+\frac{k}{\ell r}+\frac{k}{\ell^2 r}+\cdots\\
&= \frac{k}{r} \left( v_\ell(q^r-1)+\frac{1}{\ell-1}+1\right)+\frac{1}{r},
\end{align*}
which concludes the proof.
\end{proof}
\begin{proof}[Proof of Theorem \ref{quasi-scalar-eigenvectors}]
We claim that we may take $r_\alpha=2C(\alpha, \ell, 1)$, defined as in Lemma \ref{q-power-sum-bound}.  We first show that for $r>r_\alpha$, every $\sigma_\alpha$-eigenvector in $\mathbb{Q}_\ell[[\pi_1^\ell(X_{\bar k}, \bar x)]]$ in fact lies in $\mathbb{Q}_\ell[[\pi_1^\ell(X_{\bar k}, \bar x)]]^{\leq \ell^{-r}}$. 

  We claim that there exist unique $\sigma_\alpha$-equivariant splittings $s_i$ of the quotient maps $$W^{-i}\mathbb{Q}_\ell[[\pi_1^\ell(X_{\bar k}, \bar x)]]^{\leq \ell^{-r}}\to \on{gr}^{-i}_W\mathbb{Q}_\ell[[\pi_1^\ell(X_{\bar k}, \bar x)]]^{\leq \ell^{-r}},$$ for any $r>r_\alpha$; this will imply that any $\sigma_\alpha$-eigenvector in  $\mathbb{Q}_\ell[[\pi_1^\ell(X_{\bar k}, \bar x)]]$ in fact lies in $\mathbb{Q}_\ell[[\pi_1^\ell(X_{\bar k}, \bar x)]]^{\leq \ell^{-r}}$.  Indeed, by Lemma \ref{l-adic-periods-bound} and the Yoneda interpretation of $\on{Ext}^1$ in terms of extension classes, there exists a $\sigma_\alpha$-equivariant map $$s_i^m:\on{gr}^{-i}_W\mathbb{Z}_\ell[[\pi_1^\ell(X_{\bar k}, \bar x)]]\to W^{-i}\mathbb{Z}_\ell[[\pi_1^\ell(X_{\bar k}, \bar x)]]/W^{-m}\mathbb{Z}_\ell[[\pi_1^\ell(X_{\bar k}, \bar x)]]$$ such that the diagram
$$\xymatrix{
		&					&					& \on{gr}^{-i}_W\mathbb{Z}_\ell[[\pi_1^\ell(X_{\bar k}, \bar x)]] \ar[d]^{\cdot\ell^{v(i,m,\alpha)}}\ar[ld]_{s_i^m} & \\
0\ar[r] 	& W^{-i-1}/W^{-m}\ar[r] 	& W^{-i}/W^{-m} \ar[r] 	& \on{gr}^{-i}_W\mathbb{Z}_\ell[[\pi_1^\ell(X_{\bar k}, \bar x)]]\ar[r] 	&0
}$$
commutes, where $W^{-i}$ above denotes $W^{-i}\mathbb{Z}_\ell[[\pi_1^\ell(X_{\bar k}, \bar x)]]$, and $v(i,m,\alpha)$ is defined as in Lemma \ref{l-adic-periods-bound}.  Moreover $s_i^m$ is unique because $\alpha$ does not appear as an eigenvalue for the action of $\sigma_\alpha$ on $W^{-i-1}/W^{-m}$.  By uniqueness, the diagram 
$$\xymatrix{
 &&\on{gr}^{-i}_W\mathbb{Z}_\ell[[\pi_1^\ell(X_{\bar k}, \bar x)]]  \ar[d]^{\ell^{-v(i,m,\alpha)}\cdot s_i^m} \ar[lld]_{\ell^{-v(i,m+1,\alpha)}\cdot s_i^{m+1}} &\\
W^{-i}\mathbb{Q}_\ell[[\pi_1^\ell(X_{\bar k}, \bar x)]]/W^{-m-1}\mathbb{Q}_\ell[[\pi_1^\ell(X_{\bar k}, \bar x)]] \ar[rr] & & W^{-i}\mathbb{Q}_\ell[[\pi_1^\ell(X_{\bar k}, \bar x)]]/W^{-m}\mathbb{Q}_\ell[[\pi_1^\ell(X_{\bar k}, \bar x)]]
}$$
commutes, so, extending scalars, the maps $\{\ell^{-v(i,m,\alpha)}\cdot s_i^m\}_{m>i}$ define (by completeness of $\mathbb{Q}_\ell[[\pi_1^\ell(X_{\bar k}, \bar x)]]$ with respect to the $\mathscr{I}$-adic, hence $W^\bullet$ filtration) a $\sigma_\alpha$-equivariant map $$\tilde s_i: \on{gr}^{-i}_W\mathbb{Q}_\ell[[\pi_1^\ell(X_{\bar k}, \bar x)]] \to W^{-i}\mathbb{Q}_\ell[[\pi_1^\ell(X_{\bar k}, \bar x)]],$$ splitting the natural quotient map $W^{-i}\to \on{gr}^{-i}_W$.  We claim this map factors through $\mathbb{Q}_\ell[[\pi_1^\ell(X_{\bar k}, \bar x)]]^{\leq \ell^{-r}}$ for $r> 2C(\alpha, \ell, 1)$, giving the desired maps $s_i$.

But indeed, for $y\in \on{gr}^{-i}_W\mathbb{Z}_\ell[[\pi_1^\ell(X_{\bar k}, \bar x)]]$, $$v_n(\pi_n(\tilde s_i(y)))\geq - v(i, 2n, \alpha)$$ by definition of the $\tilde s_i$ and by Proposition \ref{W-adic-topology-prop} (using that $W^{-2n-1}\subset \mathscr{I}^n$).  But by Lemma \ref{q-power-sum-bound}, $$v(i, 2n, \alpha)\leq C(\alpha,\ell,2n-i-1);$$ and $C(\alpha, \ell, 2n-i-1)-nr\to \infty$ as $n\to \infty$, as we've chosen $r>2C(\alpha, \ell,1)$.  Thus $$v_n(\pi_n(\tilde s_i(y)))+nr\to \infty$$ and $\tilde s_i$ factors though $\mathbb{Q}_\ell[[\pi_1^\ell(X_{\bar k}, \bar x)]]^{\leq \ell^{-r}}$ as desired.  Let $$s_i: \on{gr}^{-i}_W\mathbb{Q}_\ell[[\pi_1^\ell(X_{\bar k}, \bar x)]] \to W^{-i}\mathbb{Q}_\ell[[\pi_1^\ell(X_{\bar k}, \bar x)]]^{\leq \ell^{-r}}$$ be the induced map.

 Now let $y\in \mathbb{Q}_\ell[[\pi_1^\ell(X_{\bar k}, \bar x)]]$ be a non-zero $\sigma_\alpha$-eigenvector; then there exists some maximal $i$ such that $y\in W^{-i}$, as the $W$-adic topology is separated by Proposition \ref{W-adic-topology-prop}.  Then we have $\sigma_\alpha y=\alpha^iy$, because $y\not=0 \bmod W^{-i-1}$. Now we claim $y=s_i(y \bmod W^{-i-1})$; indeed, $y-s_i(y \bmod W^{-i-1})$ is a $\sigma_\alpha$-eigenvector with eigenvalue $\alpha^i$, contained in $W^{-i-1}$, hence zero.  Thus $y$ is in $W^{-i}\mathbb{Q}_\ell[[\pi_1^\ell(X_{\bar k}, \bar x)]]^{\leq \ell^{-r}}$, as desired.
 
We now check that the $\sigma_\alpha$-eigenvectors in $W^{-i}\mathbb{Q}_\ell[[\pi_1^\ell(X_{\bar k}, \bar x)]]^{\leq \ell^{-r}}$ are dense in the topology defined by the Gauss norm.  Given $z\in W^{-i}\mathbb{Q}_\ell[[\pi_1^\ell(X_{\bar k}, \bar x)]]^{\leq \ell^{-r}}$, let $z_0=z, w_0=s_i(z_0\bmod W^{-i-1})$ and in general,  $$z_j=z_{j-1}-w_{j-1}, w_j=s_{i+j}(z_j\bmod W^{-i-j-1}).$$  Then the $w_j$ are $\sigma_\alpha$-eigenvectors, so it suffices to show that $$\sum w_j\to z,$$ or equivalently that $|z_n|_r\to 0$.  We leave this as an elementary exercise, again using Lemma \ref{q-power-sum-bound}.
\end{proof}
\begin{remark}\label{r-alpha-remark}
Note that we can take $r_\alpha=2C(\alpha, \ell, 1)$, defined as in Lemma \ref{q-power-sum-bound}, by the very first line in the proof above.  If $H^1(X_{\bar k}, \mathbb{Q}_\ell)$ is pure of weight $i$ ($i=1,2$), we may take $r_\alpha=C(\alpha^i, \ell,1)$ (as in this case the weight filtration equals the $\mathscr{I}$-adic filtration, up to renumbering).
\end{remark}



\section{Applications to arithmetic and geometric representations}\label{applications-section}
\subsection{Main Results}
We now prepare for the proof of Theorem \ref{main-arithmetic-result}.
\begin{lemma}\label{subquotient-lemma}
Suppose that $$\rho: \pi_1^{\text{\'et}}(X_{\bar k}, \bar x)\to GL_n(\mathbb{Z}_\ell)$$ is an arithmetic representation (defined as in Definition \ref{arithmetic-defn}), which is trivial mod $\ell^r$, and that $\rho\otimes \mathbb{Q}_\ell$ is irreducible.  Then there is a finite extension $k\subset k'$ and a representation $\beta: \pi_1^{\text{\'et}}(X_{k'}, \bar x)\to GL_n(\mathbb{Z}_\ell)$ such that 
\begin{itemize}
\item $\rho$ is a subquotient of $\beta|_{\pi_1^{\text{\'et}}(X_{\bar k}, \bar x)}$, and
\item $\beta|_{\pi_1^{\text{\'et}}(X_{\bar k}, \bar x)}$ is trivial mod $\ell^r$.
\end{itemize}
\end{lemma}
\begin{proof}
By assumption, there exists a finite extension $k\subset k'$ and a representation $$\gamma: \pi_1^{\text{\'et}}(X_{k'}, \bar x)\to GL_n(\mathbb{Z}_\ell)$$ such that $\rho$ arises as a subquotient of $\gamma|_{\pi_1^{\text{\'et}}(X_{\bar k}, \bar x)}$.  Let $0=V^0\subset V^1 \subset \cdots V^r=\gamma$ be the socle filtration of $\gamma|_{\pi_1^{\text{\'et}}(X_{\bar k}, \bar x)}\otimes\mathbb{Q}_\ell$ (viewed as a representation of $\pi_1^{\text{\'et}}(X_{\bar k}, \bar x)$), i.e. $V^i$ is the largest subrepresentation of $\gamma|_{\pi_1^{\text{\'et}}(X_{\bar k}, \bar x)}$ such that $V^i/V^{i-1}$ is semi-simple as a $\pi_1^{\text{\'et}}(X_{\bar k}, \bar x)$-representation.  As the socle filtration is canonical, each $V^i/V^{i-1}$ also extends to a representation of $\pi_1^{\text{\'et}}(X_{k'}, \bar k)$.  As $\rho\otimes \mathbb{Q}_\ell$ is irreducible, it arises as a direct summand of $V^i/V^{i-1}$ for some $i$ (viewed as a $\pi_1^{\text{\'et}}(X_{\bar k}, \bar x)$-representation).  Thus by replacing $\gamma$ with $(V^i\cap \gamma)/(V^{i-1}\cap \gamma)$, we may assume $\gamma|_{\pi_1^{\text{\'et}}(X_{\bar k}, \bar x)}\otimes \mathbb{Q}_\ell$ is semisimple.

Let $W\subset V^i/V^{i-1}$ be the minimal sub-$\pi_1^{\text{\'et}}(X_{k'}, \bar x)$-representation containing $\rho\otimes \mathbb{Q}_\ell$.  As a $\pi_1^{\text{\'et}}(X_{\bar k}, \bar x)$-representation, $W$ splits as a finite direct sum of the form $$W=\bigoplus (\rho^{\sigma}\otimes \mathbb{Q}_\ell)^{\alpha_\sigma},$$ where $\rho^{\sigma}$ denotes the representation obtained by pre-composing $\rho$ with an automorphism $\sigma$ of $\pi_1^{\text{\'et}}(X_{\bar k}, \bar x)$.  Let $\beta$ be the $\mathbb{Z}_\ell$-submodule of $W$ spanned by $\rho^\tau$, for $\tau\in \pi_1^{\text{\'et}}(X_{k'}, \bar x).$  $\beta$ is a lattice by the compactness of $G_k$, and $\rho$ is clearly a subquotient of $\beta$.  Moreover $\beta$ is spanned by the modules $\rho^\tau$, each of which is trivial mod $\ell^r$ (as representations of $\pi_1^{\text{\'et}}(X_{\bar k}, \bar x)$), and hence is trivial mod $\ell^r$ as desired.  
\end{proof}

\begin{proof}[Proof of Theorem \ref{main-arithmetic-result}]
By \cite{deligne2}, there exists a curve $C/k$ and a map $C\to X$ such that $$\pi_1^{\text{\'et}}(C_{\bar k}, \bar x)\to \pi_1^{\text{\'et}}(X_{\bar k}, \bar x)$$ is surjective for any geometric point $\bar x$ of $C$. We may assume $C$ is affine by deleting a point, which does not affect the surjectivity of the map in question.  Thus we may immediately replace $X$ with $C$; observe that $\pi_1^\ell(C_{\bar k}, \bar x)$ is a finitely-generated free pro-$\ell$ group and in particular satisfies Condition \ref{condition-star}.  After replacing $k$ with a finite extension $k'$, we may assume $\bar x$ comes from a rational point of $C$, giving a natural action of $G_{k'}$ on $\pi_1^{\ell}(C_{\bar k}, \bar x)$.    Thus we replace $k$ with $k'$ and $X$ with $C$ (and we rename $C$ as $X$ in what follows).

By Theorem \ref{quasi-scalar-pi1}, there exists $\alpha\in \mathbb{Z}_\ell^\times$, not a root of unity, and $\sigma_\alpha \in G_k$ such that $\sigma_\alpha$ acts on $\on{gr}^{-i}_W\mathbb{Q}_\ell[[\pi_1^{\text{\'et}, \ell}(X_{\bar k}, \bar x)]]$ via $\alpha^i\cdot \on{Id}$.  Then by Theorem \ref{quasi-scalar-eigenvectors}, there exists $r_\alpha>0$ such that the $\sigma_\alpha$-action on $\mathbb{Q}_\ell[[\pi_1^{\ell}(X_{\bar k}, \bar x)]]^{\leq \ell^{-r}}$ admits a set of eigenvectors with span dense in the $r$-Gauss norm topology, for any $r>r_\alpha$.  Set $N$ to be the least integer greater than $r_\alpha$, and choose $r$ with $r_\alpha<r<N$.

Now let $$\rho:\pi_1^{\text{\'et}}(X_{\bar k}, \bar x)\to GL_n(\mathbb{Z}_\ell)$$ be an arithmetic representation, as in the statement of the theorem, such that $\rho$ is trivial mod $\ell^N$.  We claim that $\rho$ is in fact unipotent.  Let $0=V_0\subset V_1\subset \cdots$ be the socle filtration of $\rho\otimes\mathbb{Q}_\ell$; then replacing $\rho$ with $(\rho \cap V_i)/(\rho \cap V_{i-1})$ (which is arithmetic, as it is a subquotient of an arithmetic representation), we may assume $\rho\otimes \mathbb{Q}_\ell$ is semisimple, say $\rho\otimes \mathbb{Q}_\ell=\oplus_i W_i$ with $W_i$ irreducible.  Replacing $\rho$ with $\rho\cap W_i$ for any $i$, we may assume $\rho\otimes \mathbb{Q}_\ell$ is irreducible.  Now we may apply Lemma \ref{subquotient-lemma}, so we may (at the cost of losing irreducibility), assume that there exists a finite extension $k'/k$ and a continuous representation $$\beta: \pi_1^{\text{\'et}}(X_{k'}, \bar x)\to GL_n(\mathbb{Z}_\ell)$$ such that $\beta|_{\pi_1^{\text{\'et}}(X_{\bar k}, \bar x)}$ is trivial mod $\ell^N$, and such that $\rho$ is a subquotient of $\beta|_{\pi_1^{\text{\'et}}(X_{\bar k}, \bar x)}$.  It suffices to show that $\beta|_{\pi_1^{\text{\'et}}(X_{\bar k}, \bar x)}$ is unipotent.

Let $m$ be such that $\sigma_\alpha^m\in G_{k'}$.

As $$\ker(GL_n(\mathbb{Z}_\ell)\to GL_n(\mathbb{Z}/\ell^N\mathbb{Z}))$$ is a pro-$\ell$-group, $\beta|_{\pi_1^{\text{\'et}}(X_{\bar k}, \bar x)}$ factors through $\pi_1^{\ell}(X_{\bar k}, \bar x)$.  Thus by Proposition \ref{convergence-prop}, $\beta$ induces a $\sigma_\alpha^m$-equivariant map $$\mathbb{Q}_\ell[[\pi_1^{\ell}(X_{\bar k}, \bar x)]]^{\leq \ell^{-r}}\overset{\tilde \beta}{\longrightarrow} \mathfrak{gl}(\mathbb{Q}_\ell),$$ where the $\sigma_\alpha^m$-action on $\pi_1^{\ell}(X_{\bar k}, \bar x)$ comes from our choice of rational basepoint, and the action on $\mathfrak{gl}_n(\mathbb{Q}_\ell)$ comes from the the fact that $\beta|_{\pi_1^{\text{\'et}}(X_{\bar k}, \bar x)}$ extends to a representation of $\pi_1^{\text{\'et}}(X_{k'}, \bar x)$, by the previous paragraph.  

Now $\mathfrak{gl}_n(\mathbb{Q}_\ell)$ is a finite-dimensional vector space, so the action of $\sigma_{\alpha}^m$ on $\mathfrak{gl}_n(\mathbb{Q}_\ell)$ only has finitely many eigenvalues.  In particular, for $s\gg 0, \alpha^s$ is not an eigenvalue of the  $\sigma_{\alpha}^m$-action on $\mathfrak{gl}_n(\mathbb{Q}_\ell)$.  Thus by Theorem \ref{quasi-scalar-eigenvectors}, $$\tilde \beta(W^{-i}\mathbb{Q}_\ell[[\pi_1^{\ell}(X_{\bar k}, \bar x)]]^{\leq \ell^{-r}})=0$$ for $i\gg0$, as $W^{-i}\mathbb{Q}_\ell[[\pi_1^{\ell}(X_{\bar k}, \bar x)]]^{\leq \ell^{-r}}$ admits a set of $\sigma_\alpha$-eigenvectors with dense span and with associated eigenvalues in $\{\alpha^i, \alpha^{i+1}, \cdots\}$.  Hence in particular $\tilde \beta(W^{-i}\mathbb{Z}_\ell[[\pi_1^{\ell}(X_{\bar k}, \bar x)]])=0$, where we may consider $\mathbb{Z}_\ell[[\pi_1^{\ell}(X_{\bar k}, \bar x)]]$ as a sub-algebra of $\mathbb{Q}_\ell[[\pi_1^{\ell}(X_{\bar k}, \bar x)]]$ because Condition \ref{condition-star} is satisfied by the first paragraph of this proof.

But by Proposition \ref{W-adic-topology-prop} (and again using Condition \ref{condition-star}), the $W$-adic topology on $\mathbb{Z}_\ell[[\pi_1^{\ell}(X_{\bar k}, \bar x)]]$ is the same as the $\mathscr{I}$-adic topology --- hence $\tilde\beta(\mathscr{I}^t)=0$ for $t\gg 0$.  But this immediately implies that $\beta|_{\pi_1^{\text{\'et}}(X_{\bar k}, \bar x)}$, and hence $\rho$, is unipotent as desired.
\end{proof}
\begin{remark}\label{bound-observation-2}
Observe that $N$ only depends on $X$ and $\ell$, and not on any of the parameters of $\rho$ (for example, its dimension).  To our knowledge this was not expected.  Indeed, if $X$ is an affine curve, $N(X, \ell)$ only depends on the index of the image of the natural Galois representation $$G_k\to GL(H^1(X_{\bar k}, \mathbb{Z}_\ell))$$ in its Zariski-closure, by the proof of Theorem \ref{quasi-scalar-pi1}.  If this index is $1$ for some $\ell>2$ (as is expected to be the case for almost all $\ell$ (see e.g. \cite[\S 10]{serreladic}, \cite{wintenberger}, and \cite{cadoretmoonen} for discussion of the case where $H^1(X_{\bar k}, \mathbb{Z}_\ell$) is pure of weight $1$ --- as far as we know, the mixed case has not been conjectured in the literature, though it seems natural to do so), we may take $N(X, \ell)=1$, by choosing $\alpha$ in Theorem \ref{quasi-scalar-pi1} to be a topological generator of $\mathbb{Z}_\ell^\times$, and using Remark \ref{r-alpha-remark}.
\end{remark}
We may now prove Corollary \ref{main-geometric-cor}.

\begin{proof}[Proof of Corollary \ref{main-geometric-cor}]
Let $k_0\subset k$ be a finitely generated subfield over which $X$ is defined; suppose that $\rho$ arises as a subquotient of the monodromy representation on $R^i\pi_*\underline{\mathbb{Z}_\ell}$ for some smooth proper morphism $\pi: Y\to X$.  Then there exists a finitely generated $k_0$-algebra $R\subset k$, a proper $R$-scheme $\mathcal{Y}$, and a smooth proper map of $R$-schemes $$\tilde \pi: \mathcal{Y}\to X_R$$ such that $\pi$ is the base change  of $\tilde\pi$ along the inclusion $R\subset k$.   Quotienting out torsion we may replace $R^i\tilde{\pi}_*\underline{\mathbb{Z}_\ell}$ by a lisse subquotient.  Now specializing to any closed point of $\on{Spec}(R)$, we see that $\rho|_{\pi_1(X_{\bar k})}$ is arithmetic, so we are done by Theorem \ref{main-arithmetic-result}.

Observe that the $N$ from Theorem \ref{main-arithmetic-result} only depends on our choice of model of $X$ over $k_0$, which do not depend on $\rho$ in any way, so we have the desired uniformity in $\rho$.
\end{proof}

\subsection{Explicit Bounds}\label{bounds-section} In some cases, one can make the $N$ appearing in Theorem \ref{main-arithmetic-result} explicit.  Indeed, if $C$ is an affine curve $N(C, \ell)$ only depends on the index of the representation $$G_k\to GL(H^1(C_{\bar k}, \mathbb{Z}_\ell))$$ in its Zariski closure (see Remark \ref{bound-observation-2}).  Thus we now turn to cases where we understand this representation.

Our main example will be the case $$X=\mathbb{P}^1_k\setminus \{x_1, \cdots, x_n\},$$ where $x_1,\cdots, x_n\in \mathbb{P}^1(k)$.  Unlike Theorem \ref{main-arithmetic-result}, this theorem works in arbitrary characteristic.

\begin{theorem}\label{p1-main-theorem}
Let $k$ be a finitely generated field, and let $$X=\mathbb{P}^1_{\bar k}\setminus\{x_1, \cdots, x_n\},$$ for $x_1, \cdots, x_n\in \mathbb{P}^1(k)$.     Let $\ell$ be prime different from the characteristic of $k$ and $q\in \mathbb{Z}_\ell^\times$ any element of the  image of the cyclotomic character $$\chi: \on{Gal}(\overline{k}/k)\to \mathbb{Z}_\ell^\times;$$ let $s$ be the order of $q$ in $\mathbb{F}_\ell^\times$ if $\ell\not=2$ and in $(\mathbb{Z}/4\mathbb{Z})^\times$ if $\ell=2$.  Let $\epsilon=1$ if $\ell=2$ and $0$ otherwise.  Let $$\rho: \pi_1^{\text{\'et}}(X_{\bar k})\to GL_m(\mathbb{Z}_\ell)$$ be an arithmetic representation such that $\rho|_{\pi_1(X_{\bar k}, \bar x)}$ is trivial mod $\ell^N$, with  $$N>\frac{1}{s}\left(v_\ell(q^s-1)+\frac{1}{\ell-1}+\epsilon\right).$$  Then  $\rho$ is unipotent. 
\end{theorem}

\begin{remark}\label{dumb-bounds-remark}
For any $X$, the bound above becomes $$N>\frac{1}{\ell-1}+\frac{1}{(\ell-1)^2}$$ for odd $\ell$, if $q$ is a topological generator of $\mathbb{Z}_\ell^\times$.  And we may choose $q$ to be a topological generator if $\ell>2$ and the cyclotomic character $$\chi: \on{Gal}(\overline{k}/k)\to \mathbb{Z}_\ell^\times$$ is surjective. Thus in particular if $\chi$ is surjective at $\ell$ for some odd $\ell$ (note that this holds for almost all $\ell$), then any arithmetic representation of $\pi_1^{\text{\'et}}(X_{\bar k})$, which is trivial mod $\ell$, is unipotent. 
\end{remark}
\begin{proof}[Proof of Theorem \ref{p1-main-theorem}]  
We first choose a rational (or rational tangential) basepoint $x$ of $X$ over $k$; we may always choose a rational tangential basepoint as the $x_i\in \mathbb{P}^1(k)$.  Let $\sigma_q\in G_k$ be such that $\chi(\sigma_q)=q$.  Choose $r$ with $$N>r>\frac{1}{s}\left(v_\ell(q^s-1)+\frac{1}{\ell-1}+\epsilon\right)=C(q,\ell,1),$$
where $C(q,\ell,1)$ is defined as in Lemma \ref{q-power-sum-bound}.

In this case, $$H^1(X, \mathbb{Z}_\ell)=\mathbb{Z}_\ell(1)^{n-1},$$ so $\sigma_q$ satisfies the conclusions of Theorem \ref{quasi-scalar-pi1}.  Thus by Theorem \ref{quasi-scalar-eigenvectors} and Remark \ref{r-alpha-remark} (using that $r>C(q,\ell,1)$), $W^{-i}\mathbb{Q}_\ell[[\pi_1^\ell(X_{\bar k}, \bar x)]]^{\leq \ell^{-r}}$ admits a set of $\sigma_q$-eigenvectors with dense span in the Gauss norm topology, and with eigenvalues in $\{q^i, q^{i+1}, \cdots\}$.  Now we conclude by precisely imitating the proof of Theorem \ref{main-arithmetic-result}.
\end{proof}
\begin{corollary}\label{main-theorem-P1}
Let $k$ be a field with prime subfield $k_0$, and let $$X=\mathbb{P}^1_{\bar k}\setminus\{x_1, \cdots, x_n\}.$$ Let $K$ be the field generated by cross-ratios of the $x_i$, that is, $$K=k_0\left(\frac{x_a-x_b}{x_c-x_d}\right)_{1\leq a<b<c<d\leq n}.$$   Let $\ell$ be a prime different from the characteristic of $k$ and $q\in \mathbb{Z}_\ell^\times$ be any element of the image of the cyclotomic character $$\chi: \on{Gal}(\overline{K}/K)\to \mathbb{Z}_\ell^\times;$$ let $s$ be the order of $q$ in $\mathbb{F}_\ell^\times$ if $\ell\not=2$ and in $(\mathbb{Z}/4\mathbb{Z})^\times$ if $\ell=2$.  Let $\epsilon=1$ if $\ell=2$ and $0$ otherwise.  Let $$\rho: \pi_1^{\text{\'et}}(X, x)\to GL_m(\mathbb{Z}_\ell)$$ be a continuous representation which is trivial mod $\ell^N$, with $$N>\frac{1}{s}\left(v_\ell(q^s-1)+\frac{1}{\ell-1}+\epsilon\right).$$  If $\rho$ is geometric, then it is unipotent.  
\end{corollary}
\begin{proof}
We let $X=\mathbb{P}^1_{k}\setminus\{x_1, \cdots, x_n\}$.  Then $X$ admits a natural model over the field $$K=k_0\left(\frac{x_a-x_b}{x_c-x_d}\right)_{1\leq a<b<c<d\leq n}.$$ Now the result follows from Theorem \ref{p1-main-theorem} in a manner identical to the deduction of Corollary \ref{main-geometric-cor} from Theorem \ref{main-arithmetic-result}.
\end{proof}
\begin{remark}\label{dumb-bounds-remark-2}
If $k=\mathbb{Q}$ in Theorem \ref{p1-main-theorem}, or if $K=\mathbb{Q}$ in Corollary \ref{main-theorem-P1}, we may take $N(X,\ell)=1$ for any odd $\ell$, by Remark \ref{dumb-bounds-remark}.
\end{remark}
We give one final corollary over fields of arbitrary characteristic.  Observe that if there is a map $\mathbb{P}^1\setminus\{x_1, \cdots, x_n\}\to X$ which induces a surjection on geometric pro-$\ell$ fundamental groups, we may apply Theorem \ref{main-theorem-P1} to find restrictions on arithmetic representations of the fundamental group of $X$.  Such a map exists if $X$ is an open subvariety of a separably rationally connected variety, by the main result of \cite{kollar}.  Thus we have
\begin{corollary}
Let $k$ be a finitely generated field an $X$ an open subset of a separably rationally connected $k$-variety.  Let $\ell$ be a prime different from the characteristic of $k$.  Then there exists $N=N(X,\ell)$ such that if $$\rho:\pi_1^{\text{\'et}}(X_{\bar k})\to GL_m(\mathbb{Z}_\ell)$$ is arithmetic, and is trivial mod $\ell^N$, then $\rho$ is unipotent.
\end{corollary}
\subsection{Sharpness of results and examples}\label{sharp-section}

The hypothesis on the cyclotomic character in Theorem \ref{p1-main-theorem} and Corollary \ref{main-theorem-P1} may seem strange, but they are in fact necessary.
\begin{example}
Let $\ell=3$ or $5$, and consider the (connected) modular curve $Y(\ell)$, parametrizing elliptic curves with full level $\ell$ structure.  $Y(\ell)$ has genus zero.  Let $E\to Y(\ell)$ be the universal family, and $\bar x$ any geometric point of $Y(\ell)$.  Then the tautological representation $$\rho: \pi_1^{\text{\'et}}(Y(\ell), \bar x)\to GL(T_\ell(E_{\bar x}))$$ is trivial mod $\ell$, and $\rho|_{\pi_1^{\text{\'et}}(Y(\ell)_{\bar k}, \bar x)
}$ is evidently both arithmetic and geometric.  This does not contradict Theorem \ref{p1-main-theorem}, Corollary \ref{main-theorem-P1}, or Remark \ref{dumb-bounds-remark-2} because the field generated by the cross-ratios of the cusps of $Y(\ell)$ is $\mathbb{Q}(\zeta_\ell)$, whose cyclotomic character is not surjective at $\ell$.
\end{example}
We may use these results to construct example of representations of fundamental groups which do not come from geometry --- indeed, the following is an example of a representation which is not arithmetic or geometric, and which we do not know how to rule out by other means:
\begin{example}\label{ridic-example}
As before, let $Y(3)$ be the modular curve parametrizing elliptic curves with full level three structure.  Then $$Y(3)_{\mathbb{C}}\simeq \mathbb{P}^1\setminus\{0,1,\infty, \lambda\}$$ where $\lambda\in \mathbb{Q}(\zeta_3)$; let $x\in Y(3)(\mathbb{C})$ be a point.  Let $\rho$ be the tautological representation $$\rho: \pi_1(Y(3)(\mathbb{C})^{\text{an}}, x)\to GL(H^1(E_x(\mathbb{C})^{\text{an}}, \mathbb{Z})).$$  Let $$X=\mathbb{P}^1\setminus \{0,1,\infty, \beta\}$$ where $\beta\in \mathbb{Q}\setminus\{0,1\}$.  Then $X(\mathbb{C})^{\text{an}}$ is homeomorphic to $Y(3)(\mathbb{C})^{\text{an}}$ (indeed, both are homeomorphic to a four-times-punctured sphere); let $$j: X(\mathbb{C})^{\text{an}}\to Y(3)(\mathbb{C})^{\text{an}}$$ be such a homeomorphism.  Then the representation $$\tilde{\rho}: \pi_1(X(\mathbb{C})^{\text{an}}, j^{-1}(x))\overset{j_*}{\longrightarrow} \pi_1(Y(3)(\mathbb{C})^{\text{an}}, x)\overset{\rho}{\longrightarrow}  GL(H^1(E_x(\mathbb{C})^{\text{an}}, \mathbb{Z}))$$ is trivial mod $3$ and thus cannot arise from geometry, by Remark \ref{dumb-bounds-remark-2}.  Likewise, the representation $$\pi_1^{\text{\'et}}(X)\to GL(H^1(E_x, \mathbb{Z}_\ell))$$ obtained from $\tilde \rho$ cannot be arithmetic, again by Remark \ref{dumb-bounds-remark-2}.  

We do not know a way to see this using pre-existing methods, since any criterion ruling out this representation would have to detect the difference between $X$ and $Y(3)$; for example, the quasi-unipotent local monodromy theorem does not rule out $\tilde \rho$.  This example was suggested to the author by George Boxer.
\end{example}
\begin{example}
Let $X/\mathbb{Q}$ be a proper genus two curve; recall that $$\pi_1(X(\mathbb{C})^{\text{an}})=<a_1, b_1, a_2, b_2>/([a_1,b_1][a_2,b_2]=1).$$  Let $\ell$ be a prime and $A, B$ non-unipotent $n\times n$ integer matrices which are equal to the identity mod $\ell^N$.  Then for $N\gg 0$, Corollary \ref{main-geometric-cor} implies that the representation $$a_1\mapsto A, b_1\mapsto B, a_2\mapsto B, b_2\mapsto A$$ does not come from geometry; likewise the induced representation $$\pi_1^{\text{\'et}}(X_{\overline{\mathbb{Q}}})\to GL_n(\mathbb{Z}_\ell)$$ is not arithmetic by Theorem \ref{main-arithmetic-result}.  Again, we do not know how to see this using previously known results.
\end{example}

On the other hand, we do not expect our bounds on $r_\alpha$ in Theorem \ref{quasi-scalar-eigenvectors} to be sharp for arbitrary $X$.  In some cases, however, one may bound $r_\alpha$ from below by finding a non-unipotent arithmetic representation of $\pi_1(X_{\bar k})$ trivial mod $\ell^M$.  In this case, by the contrapositive of Theorem \ref{main-arithmetic-result}, we have $r_\alpha\geq M$ for any $\sigma_\alpha$; for example, by considering the tautological monodromy representation of the fundamental group of the modular curve $Y(\ell^M)$, we find that $r_\alpha\geq M$ for any $\sigma_\alpha$ acting on the fundamental group of $Y(\ell^M)$.

We were also (with the exception of Theorem \ref{p1-main-theorem} and Corollary \ref{main-theorem-P1}) unable to prove results in positive characteristic, because $\ell$-adic open image theorems as in \cite{bogomolov} are no longer true in this case, so the proof of Lemma \ref{quasi-scalar-h1} fails.  We could salvage this situation if the following question has a positive answer:
\begin{question}
Let $X$ be a smooth curve over a finite field $k$, and let $\ell$ be a prime different from the characteristic of $k$.  Does there exist an $r=r(X)$ such that ${\mathbb{Q}_\ell}[[\pi_1^{\ell}(X_{\bar k}, \bar x)]]^{\leq \ell^{-r}}$ admits a set of Frobenius eigenvectors with dense span?
\end{question}
As in Theorem \ref{quasi-scalar-eigenvectors}, one would need to show that for a Frobenius eigenvector $y\in {\mathbb{Q}_\ell}[[\pi_1^{\ell}(X_{\bar k}, \bar x)]]$, one has $|v_n(\pi_n(y))|=O(n)$.  We can show, using Yu's $p$-adic Baker's theorem on linear forms in logarithms \cite{yu-baker}, that for almost all $\ell$, $|v_n(\pi_n(y))|=O(n\log n)$, which does not suffice for our applications.

One could also ask for stronger uniformity in geometric invariants of $X$.  For example, a positive answer to the following question would imply positive answer to a ``pro-$\ell$" version of the geometric torsion conjecture, by the methods of this section.
\begin{question}
Let $X$ be a smooth curve over a finite field $k$, and let $\ell$ be a prime different from the characteristic of $k$.  Does there exist an $r=r(\on{gonality}(X), \ell)>0$, depending only on the gonality of $X$ and on $\ell$, such that $\overline{\mathbb{Q}_\ell}[[\pi_1^{\ell}(X_{\bar k}, \bar x)]]^{\leq \ell^{-r}}$ admits a set of Frobenius eigenvectors with dense span?  If so, does $r$ tend to zero rapidly with $\ell$?
\end{question}

\bibliographystyle{alpha}
\bibliography{monodromy-bibtex}

\end{document}